\documentclass{article}

\usepackage{color}
\newcommand{\todo}[1]{#1}

\usepackage{amsmath,amssymb,amsthm}
\newtheorem{theorem}{Theorem}
\newtheorem{lemma}[theorem]{Lemma}
\newtheorem{proposition}[theorem]{Proposition}
\newtheorem{corollary}[theorem]{Corollary}
\theoremstyle{definition}
\newtheorem{definition}[theorem]{Definition} 
 
\theoremstyle{remark}
\newtheorem{remark}[theorem]{Remark}

\usepackage{tikz}
\usetikzlibrary{calc,decorations.markings}
\pgfdeclarelayer{edgelayer}
\pgfdeclarelayer{nodelayer}
\pgfsetlayers{edgelayer,nodelayer,main}
\tikzstyle{none}=[inner sep=1pt]
\tikzstyle{None}=[inner sep=1pt, fill=white]
\tikzstyle{dashedcircle}=[circle, draw=gray, dashed, inner sep=6pt]
\tikzstyle{box}=[draw=black, fill=white, inner sep=.5ex, rounded corners=.1ex]
\tikzstyle{roundedbox}=[draw=black, fill=white, inner sep=.5ex, rounded corners=1ex]
\tikzstyle{cross}=[preaction={draw=white, -, line width=3pt}]
\tikzstyle{arrow}=[postaction=decorate]
\newcommand{\markat}{0.5}
\newcommand{\markwithsym}{>}
\newcommand{\markwith}{{\arrow[black]{\markwithsym}}}
\pgfkeys{/tikz/.cd, markat/.store in=\markat, markwith/.store in=\markwithsym} 
\tikzset{decoration={markings, mark=at position \markat with \markwith}}
\newcommand{\ground}[2]{
  \node[inner sep=0mm] (#1) at (#2) {};
  \draw[thick]  ($(#2)+(0.3,-0.01)$) -- ($(#2)+(-0.3,-0.01)$);
  \draw[thick]  ($(#2)+(0.23,0.069)$) -- ($(#2)+(-0.22,0.069)$);
  \draw[thick]  ($(#2)+(0.16,0.139)$) -- ($(#2)+(-0.16,0.139)$);
  \draw[thick]  ($(#2)+(0.09,0.209)$) -- ($(#2)+(-0.09,0.209)$);
  \draw[thick]  ($(#2)+(0.02,0.279)$) -- ($(#2)+(-0.02,0.279)$);
}

\newcommand{\after}{\ensuremath{\circ}}
\newcommand{\tensor}{\ensuremath{\otimes}}
\newcommand{\id}[1][]{\ensuremath{1_{#1}}}
\newcommand{\cat}[1]{\ensuremath{\mathbf{#1}}}
\newcommand{\Cat}[1]{\ensuremath{\mathbf{#1}}}
\newcommand{\blank}{\ensuremath{\underline{\phantom{n}}}}
  
\newcommand{\CP}{\ensuremath{\mathrm{CP}^\infty}}
\newcommand{\CPM}{\ensuremath{\mathrm{CPM}}}

\begin{document}

\title{Pictures of complete positivity in arbitrary dimension}
\author{Bob Coecke and Chris Heunen}
\maketitle
\begin{abstract}
  Two fundamental contributions to categorical quantum mechanics are presented.
  First, we generalize the CPM--construction, that turns any dagger
  compact category into one with completely positive maps, to
  arbitrary dimension. Second, we axiomatize when a given category is
  the result of this construction. 
\end{abstract}

\section{Introduction}
\label{sec:intro}

Since the start of categorical quantum
mechanics~\cite{abramskycoecke:categoricalsemantics}, dagger
compactness has played a key role in most constructions, protocol
derivations and theorems. To name two:
\begin{itemize}
\item Selinger's \CPM--construction, which associates
  to any dagger compact category of pure states and operations a corresponding dagger
  compact category of mixed states and operations~\cite{selinger:completelypositive};
\item Environment structures, an axiomatic substitute for the
  \CPM--construction which proved to be particularly useful in the
  derivation of quantum protocols~\cite{coecke:selinger, coeckeperdrix:channels}.
\end{itemize}

It is well known that assuming compactness imposes finite dimension
when exporting these results to the Hilbert space
model~\cite{heunen:compactlyaccessible}. This paper introduces
variations of each the above two results that rely on dagger structure alone, and
in the presence of compactness reduce to the above ones. Hence,
these variations accommodate interpretation not just in the dagger 
compact category of finite dimensional Hilbert spaces and linear maps,
but also in the dagger category of Hilbert spaces of arbitrary
dimension and continuous linear maps. We show:
\begin{itemize}
\item that the generalized \CPM--construction indeed corresponds to the
  usual definitions of infinite-dimensional quantum information theory;
\item that the direct correspondence between the \CPM--construction and
  environment structure (up to the so-called doubling axiom) still carries trough.
\end{itemize}

The next two sections each discuss one of our two variations in 
turn.

\paragraph{Earlier work}
The variation of the \CPM--construction relying solely on dagger structure was already publicized by one of the authors as a research report~\cite{coecke:mix}. 
Here we relate that construction to the usual setting of infinite-dimensional quantum information theory, thereby justifying it in terms of the usual model.  
An earlier version of this construction appeared in conference proceedings~\cite{coeckeheunen:qpl}. Whereas composition is not always well-defined there, it did have the advantage that the output category of the construction was automatically small if the input was. The construction here is closer to~\cite{coecke:mix}, and has the advantage that it is rigorously well-defined, but the disadvantage that the output category might be large. See the discussion after Proposition~\ref{prop:welldefined} and Remark~\ref{remark:composition}.
Additionally, we generalise the construction further than~\cite{coeckeheunen:qpl}, to braided monoidal categories that are not necessarily symmetric. 

\paragraph{Related work}  
While there are previous results dealing with the transition to a noncompact setting in some way or another, \textit{e.g.}~\cite{abramskyblutepanangaden:nuclearideals, heunen:compactlyaccessible, heunen:embedding}, what is particularly appealing about
the results in this paper is that they still allow the diagrammatic representations of braided monoidal
categories~\cite{joyalstreet:tensorcalculus, selinger:graphicallanguages}.  

\paragraph{Future work} 
Classical information can be modelled in categorical quantum mechanics
using so-called classical structures~\cite{coeckepavlovic:classicalobjects,CoeckePaqPav:structuralism,abramskyheunen:hstar}. It is not clear whether
these survive \CPM-like constructions; see also~\cite{heunenboixo:cpfrob}. The environment structures of Section~\ref{sec:environment} could be a useful tool in this investigation. 

\section{Complete positivity}
\label{sec:cp}

Compact categories and their graphical calculus originated in \cite{kelly, kellylaplaza:compactcategories}. For a gentle introduction to dagger (compact) categories \cite{abramskycoecke:categoricalsemantics} and their graphical calculus \cite{selinger:completelypositive}, we refer
to~\cite{CatsII}.  We now recall the \emph{\CPM--construction}~\cite{selinger:completelypositive}, that,
given a dagger compact category $\cat{C}$, produces a new dagger   
compact category $\CPM(\cat{C})$ as follows. 
When wires of both types $A$ and
$A^*$ arise in one diagram, we will decorate them with arrows in
opposite directions. When possible we will suppress coherence
isomorphisms in formulae. Finally, recall that
$(\blank)_*$ reverses the order of tensor products, so $f_*$  
has type $A^* \to B^* \tensor C^*$ when $f \colon A \to C \tensor
B$~\cite{selinger:completelypositive}.  

\begin{itemize}
\item The objects of $\CPM(\cat{C})$ are the same as those of $\cat{C}$.
\item The morphisms $A \to B$ of $\CPM(\cat{C})$ are those morphisms of
  $\cat{C}$ that can be written in the form $(\id \tensor \eta^\dag \tensor \id)(f_*
  \tensor f) \colon A^* \tensor A \to B^* \tensor B$ for some morphism
  $f \colon A \to X \tensor B$ and object $X$ in $\cat{C}$.
  \[
    \CPM(\cat{C})(A,B) = \left\{\left.
      \vcenter{\hbox{\begin{tikzpicture}[scale=0.75]
	\begin{pgfonlayer}{nodelayer}
		\node [style=none] (0) at (-1, 1) {};
		\node [style=none] (1) at (1, 1) {};
		\node [style=none] (2) at (-1, 0) {};
		\node [style=box] (3) at (-0.75, 0) {$\;\;f_*\;\;$};
		\node [style=none] (4) at (-0.5, 0) {};
		\node [style=none] (5) at (0.5, 0) {};
		\node [style=box] (6) at (0.75, 0) {$\;\;\;f\;\;$};
		\node [style=none] (7) at (1, 0) {};
		\node [style=none] (8) at (-0.75, -1) {};
		\node [style=none] (9) at (0.75, -1) {};
	\end{pgfonlayer}
	\begin{pgfonlayer}{edgelayer}
		\draw [arrow] (3) to (8);
		\draw [arrow, bend right=90, looseness=2.25] (5) to (4);
		\draw [arrow] (0) to (2);
		\draw [arrow, markat=0.66] (9) to (6);
		\draw [arrow, markat=0.66] (7) to (1);
	\end{pgfonlayer}
      \end{tikzpicture}}}
      \;\;\right|\;\;
      \vcenter{\hbox{\begin{tikzpicture}[scale=0.75]
	\begin{pgfonlayer}{nodelayer}
		\node [style=none] (0) at (-1, 1) {};
		\node [style=none] (1) at (-0.5, 1) {};
		\node [style=none] (2) at (-1, 0) {};
		\node [style=box] (3) at (-0.75, 0) {$\;\;f\;\;$};
		\node [style=none] (4) at (-0.5, 0) {};
		\node [style=none] (5) at (-0.75, -1) {};
	\end{pgfonlayer}
	\begin{pgfonlayer}{edgelayer}
		\draw [arrow, markat=0.66] (5) to (3);
		\draw [arrow, markat=0.66] (4) to (1);
		\draw [arrow, markat=0.66] (2) to (0);
	\end{pgfonlayer}
      \end{tikzpicture}}}
      \in \cat{C}(A,X \tensor B) \right\}
  \]
  We call $X$ the \emph{ancillary system} of $(\id \tensor \eta^\dag
  \tensor \id)(f_* \tensor f)$, and $f$ its \emph{Kraus morphism}; these representatives are not unique.
\item Identities are inherited from $\cat{C}$, and composition is
  defined as follows.
 \[
    \left(
      \vcenter{\hbox{\begin{tikzpicture}[scale=0.75]
	\begin{pgfonlayer}{nodelayer}
		\node [style=none] (0) at (-1, 1) {};
		\node [style=none] (1) at (1, 1) {};
		\node [style=none] (2) at (-1, 0) {};
		\node [style=box] (3) at (-0.75, 0) {$\;\;g\vphantom{f}_*\;\;$};
		\node [style=none] (4) at (-0.5, 0) {};
		\node [style=none] (5) at (0.5, 0) {};
		\node [style=box] (6) at (0.75, 0) {$\;\;\;g\vphantom{f}\;\;$};
		\node [style=none] (7) at (1, 0) {};
		\node [style=none] (8) at (-0.75, -1) {};
		\node [style=none] (9) at (0.75, -1) {};
	\end{pgfonlayer}
	\begin{pgfonlayer}{edgelayer}
		\draw [arrow] (3) to (8);
		\draw [arrow, bend right=90, looseness=2.25] (5) to (4);
		\draw [arrow] (0) to (2);
		\draw [arrow, markat=0.66] (9) to (6);
		\draw [arrow, markat=0.66] (7) to (1);
	\end{pgfonlayer}
      \end{tikzpicture}}}
   \right)
    \;\after\;
    \left(
      \vcenter{\hbox{\begin{tikzpicture}[scale=0.75]
	\begin{pgfonlayer}{nodelayer}
		\node [style=none] (0) at (-1, 1) {};
		\node [style=none] (1) at (1, 1) {};
		\node [style=none] (2) at (-1, 0) {};
		\node [style=box] (3) at (-0.75, 0) {$\;\;f_*\;\;$};
		\node [style=none] (4) at (-0.5, 0) {};
		\node [style=none] (5) at (0.5, 0) {};
		\node [style=box] (6) at (0.75, 0) {$\;\;\;f\;\;$};
		\node [style=none] (7) at (1, 0) {};
		\node [style=none] (8) at (-0.75, -1) {};
		\node [style=none] (9) at (0.75, -1) {};
	\end{pgfonlayer}
	\begin{pgfonlayer}{edgelayer}
		\draw [arrow] (3) to (8);
		\draw [arrow, bend right=90, looseness=2.25] (5) to (4);
		\draw [arrow] (0) to (2);
		\draw [arrow, markat=0.66] (9) to (6);
		\draw [arrow, markat=0.66] (7) to (1);
	\end{pgfonlayer}
      \end{tikzpicture}}}
  \right)
    \;\;=\;\;
    \vcenter{\hbox{\begin{tikzpicture}[scale=0.75]
	\begin{pgfonlayer}{nodelayer}
		\node [style=none] (0) at (-1.25, 0.75) {};
		\node [style=none] (1) at (1.25, 0.75) {};
		\node [style=none] (2) at (-1.75, 0.33) {};
		\node [style=none] (3) at (-0.1, 0.33) {};
		\node [style=none] (4) at (0.1, 0.33) {};
		\node [style=none] (5) at (1.75, 0.33) {};
		\node [style=box] (6) at (-1, 0) {$\;\;g_*\;\;$};
		\node [style=none] (7) at (-0.75, 0) {};
		\node [style=none] (8) at (-0.25, 0.2) {};
		\node [style=none] (9) at (0.25, 0.2) {};
		\node [style=none] (10) at (0.75, 0) {};
		\node [style=box] (11) at (1, 0) {$\;\;\;g\;\;$};
		\node [style=none] (12) at (-1.25, -1) {};
		\node [style=box] (13) at (-1, -1) {$\;\;f_*\;\;$};
		\node [style=none] (14) at (-0.75, -1) {};
		\node [style=none] (15) at (0.75, -1) {};
		\node [style=box] (16) at (1, -1) {$\;\;\;f\;\;$};
		\node [style=none] (17) at (1.25, -1) {};
		\node [style=none] (18) at (-1.75, -1.4) {};
		\node [style=none] (19) at (-0.1, -1.4) {};
		\node [style=none] (20) at (0.1, -1.4) {};
		\node [style=none] (21) at (1.75, -1.4) {};
		\node [style=none] (22) at (-1, -2) {};
		\node [style=none] (23) at (1, -2) {};
		\node [style=none] (24) at (-1.25, 0) {};
		\node [style=none] (25) at (1.25, 0) {};
	\end{pgfonlayer}
	\begin{pgfonlayer}{edgelayer}
		\draw [arrow] (17.center) to (25.center);
		\draw [arrow, markat=0.66] (25.center) to (1.center);
		\draw [dashed, gray] (20.center) to node{} (21.center);
		\draw [arrow, bend right=90, looseness=2.25] (10.center) to (7.center);
		\draw [dashed, gray] (18) to node{} (19);
		\draw [dashed, gray] (2) to node{} (18);
		\draw [in=90, out=270] (8) to (14.center);
		\draw [dashed, gray] (4.center) to node{} (20.center);
		\draw [arrow] (23.center) to (16);
		\draw [in=90, out=-90] (9) to (15.center);
		\draw [dashed, gray] (19) to node{} (3);
		\draw [dashed, gray] (5.center) to node{} (4.center);
		\draw [arrow] (13) to (22.center);
		\draw [arrow] (0.center) to (24.center);
		\draw [arrow] (24.center) to (12.center);
		\draw [dashed, gray] (3) to node{} (2);
		\draw [dashed, gray] (21.center) to node{} (5.center);
		\draw [arrow, bend right=90, looseness=2.50] (9.south) to (8.south);
	\end{pgfonlayer}
      \end{tikzpicture}}}
  \]
\item The tensor unit $I$ and the tensor product of objects are
  inherited from $\cat{C}$, and the tensor product of morphisms is
  defined as follows.
  \[
    \left(
      \vcenter{\hbox{\begin{tikzpicture}[scale=0.75]
	\begin{pgfonlayer}{nodelayer}
		\node [style=none] (0) at (-1, 1) {};
		\node [style=none] (1) at (1, 1) {};
		\node [style=none] (2) at (-1, 0) {};
		\node [style=box] (3) at (-0.75, 0) {$\;\;f_*\;\;$};
		\node [style=none] (4) at (-0.5, 0) {};
		\node [style=none] (5) at (0.5, 0) {};
		\node [style=box] (6) at (0.75, 0) {$\;\;\;f\;\;$};
		\node [style=none] (7) at (1, 0) {};
		\node [style=none] (8) at (-0.75, -1) {};
		\node [style=none] (9) at (0.75, -1) {};
	\end{pgfonlayer}
	\begin{pgfonlayer}{edgelayer}
		\draw [arrow] (3) to (8);
		\draw [arrow, bend right=90, looseness=2.25] (5) to (4);
		\draw [arrow] (0) to (2);
		\draw [arrow, markat=0.66] (9) to (6);
		\draw [arrow, markat=0.66] (7) to (1);
	\end{pgfonlayer}
      \end{tikzpicture}}}
   \right)
    \;\tensor\;
    \left(
      \vcenter{\hbox{\begin{tikzpicture}[scale=0.75]
	\begin{pgfonlayer}{nodelayer}
		\node [style=none] (0) at (-1, 1) {};
		\node [style=none] (1) at (1, 1) {};
		\node [style=none] (2) at (-1, 0) {};
		\node [style=box] (3) at (-0.75, 0) {$\;\;g\vphantom{f}_*\;\;$};
		\node [style=none] (4) at (-0.5, 0) {};
		\node [style=none] (5) at (0.5, 0) {};
		\node [style=box] (6) at (0.75, 0) {$\;\;\;g\vphantom{f}\;\;$};
		\node [style=none] (7) at (1, 0) {};
		\node [style=none] (8) at (-0.75, -1) {};
		\node [style=none] (9) at (0.75, -1) {};
	\end{pgfonlayer}
	\begin{pgfonlayer}{edgelayer}
		\draw [arrow] (3) to (8);
		\draw [arrow, bend right=90, looseness=2.25] (5) to (4);
		\draw [arrow] (0) to (2);
		\draw [arrow, markat=0.66] (9) to (6);
		\draw [arrow, markat=0.66] (7) to (1);
	\end{pgfonlayer}
      \end{tikzpicture}}}
   \right)
    \;\;=\;\;
    \vcenter{\hbox{\begin{tikzpicture}[scale=0.75]
	\begin{pgfonlayer}{nodelayer}
		\node [style=none] (0) at (-2.5, 1.25) {};
		\node [style=none] (1) at (-1.25, 1.25) {};
		\node [style=none] (2) at (1.25, 1.25) {};
		\node [style=none] (3) at (2.5, 1.25) {};
		\node [style=none] (4) at (-0.75, 0.75) {};
		\node [style=none] (5) at (0.75, 0.75) {};
		\node [style=None] (6) at (-3, 0.5) {};
		\node [style=None] (7) at (-0.25, 0.5) {};
		\node [style=None] (8) at (0.25, 0.5) {};
		\node [style=None] (9) at (3, 0.5) {};
		\node [style=none] (10) at (-2.5, 0) {};
		\node [style=box] (11) at (-2.25, 0) {$\;\;g_*\vphantom{f_*}\;\;$};
		\node [style=none] (12) at (-2, 0) {};
		\node [style=none] (13) at (-1.25, 0) {};
		\node [style=box] (14) at (-1, 0) {$\;\;f_*\;\;$};
		\node [style=none] (15) at (-0.75, 0) {};
		\node [style=none] (16) at (0.75, 0) {};
		\node [style=box] (17) at (1, 0) {$\;\;f\;\;$};
		\node [style=none] (18) at (1.25, 0) {};
		\node [style=none] (19) at (2, 0) {};
		\node [style=box] (20) at (2.25, 0) {$\;\;g\vphantom{f}\;\;$};
		\node [style=none] (21) at (2.5, 0) {};
		\node [style=None] (22) at (-3, -0.5) {};
		\node [style=None] (23) at (-0.25, -0.5) {};
		\node [style=None] (24) at (0.25, -0.5) {};
		\node [style=None] (25) at (3, -0.5) {};
		\node [style=none] (26) at (-2.25, -1) {};
		\node [style=none] (27) at (-1, -1) {};
		\node [style=none] (28) at (1, -1) {};
		\node [style=none] (29) at (2.25, -1) {};
	\end{pgfonlayer}
	\begin{pgfonlayer}{edgelayer}
		\draw [arrow] (28.center) to (17);
		\draw [arrow] (0.center) to (10.center);
		\draw [arrow, markat=0.66] (14) to (27.center);
		\draw [arrow] (1.center) to (13.center);
		\draw [arrow, markat=0.66] (18.center) to (2.center);
		\draw [arrow, markat=0.66] (11) to (26.center);
		\draw [arrow, markat=0.66] (21.center) to (3.center);
		\draw [arrow] (29.center) to (20);
		\draw [arrow, bend right=60] (5.center) to (4.center);
		\draw [arrow, bend right=90, looseness=1.75] (16.center) to (15.center);
		\draw [cross, in=-60, out=90] (19.center) to (5.center);
		\draw [cross, in=90, out=240] (4.center) to (12.center);
		\draw [dashed, gray] (23) to node{} (22);
		\draw [dashed, gray] (8) to node{} (24);
		\draw [dashed, gray] (6) to node{} (7);
		\draw [dashed, gray] (7) to node{} (23);
		\draw [dashed, gray] (25) to node{} (9);
		\draw [dashed, gray] (22) to node{} (6);
		\draw [dashed, gray] (9) to node{} (8);
		\draw [dashed, gray] (24) to node{} (25);
	\end{pgfonlayer}
      \end{tikzpicture}}}
  \]
\item The dagger is defined as follows.
  \[
      \left(\vcenter{\hbox{\begin{tikzpicture}[scale=0.75]
	\begin{pgfonlayer}{nodelayer}
		\node [style=none] (0) at (-1, 1) {};
		\node [style=none] (1) at (1, 1) {};
		\node [style=none] (2) at (-1, 0) {};
		\node [style=box] (3) at (-0.75, 0) {$\;\;f_*\;\;$};
		\node [style=none] (4) at (-0.5, 0) {};
		\node [style=none] (5) at (0.5, 0) {};
		\node [style=box] (6) at (0.75, 0) {$\;\;\;f\;\;$};
		\node [style=none] (7) at (1, 0) {};
		\node [style=none] (8) at (-0.75, -1) {};
		\node [style=none] (9) at (0.75, -1) {};
	\end{pgfonlayer}
	\begin{pgfonlayer}{edgelayer}
		\draw [arrow] (3) to (8);
		\draw [arrow, bend right=90, looseness=2.25] (5) to (4);
		\draw [arrow] (0) to (2);
		\draw [arrow, markat=0.66] (9) to (6);
		\draw [arrow, markat=0.66] (7) to (1);
	\end{pgfonlayer}
      \end{tikzpicture}}}
    \right)^\dag
    \;\;=\;\;
    \vcenter{\hbox{\begin{tikzpicture}[scale=0.75]
	\begin{pgfonlayer}{nodelayer}
		\node [style=none] (0) at (-2, 1) {};
		\node [style=none] (1) at (2, 1) {};
		\node [style=None] (2) at (-2.25, 0.5) {};
		\node [style=none] (3) at (-0.5, 0.5) {};
		\node [style=None] (4) at (-0.25, 0.5) {};
		\node [style=None] (5) at (0.25, 0.5) {};
		\node [style=none] (6) at (0.5, 0.5) {};
		\node [style=None] (7) at (2.25, 0.5) {};
		\node [style=none] (8) at (-2, 0) {};
		\node [style=box] (9) at (-1.5, 0) {$\;\;f^*\;\;$};
		\node [style=none] (10) at (-1.25, 0) {};
		\node [style=none] (11) at (-0.5, 0) {};
		\node [style=none] (12) at (0.5, 0) {};
		\node [style=none] (13) at (1.25, 0) {};
		\node [style=box] (14) at (1.5, 0) {$\;\;f^\dag\;\;$};
		\node [style=none] (15) at (2, 0) {};
		\node [style=None] (16) at (-2.25, -0.75) {};
		\node [style=None] (17) at (-0.25, -0.75) {};
		\node [style=None] (18) at (0.25, -0.75) {};
		\node [style=None] (19) at (2.25, -0.75) {};
		\node [style=none] (20) at (-2, -1) {};
		\node [style=none] (21) at (2, -1) {};
	\end{pgfonlayer}
	\begin{pgfonlayer}{edgelayer}
		\draw [dashed, gray] (5) to node{} (18);
		\draw (6.north) to (12.south);
		\draw [dashed, gray] (16) to node{} (17);
		\draw [bend left=270, looseness=2.25] (10) to (11);
		\draw [dashed, gray] (19) to node{} (7);
		\draw [arrow, bend right=90, looseness=1.50] (6) to (3);
		\draw [arrow, markat=0.75] (15) to (1);
		\draw [arrow] (21) to (15);
		\draw [dashed, gray] (2) to node{} (16);
		\draw [dashed, gray] (17) to node{} (4);
		\draw [dashed, gray] (18) to node{} (19);
		\draw [dashed, gray] (7) to node{} (5);
		\draw [dashed, gray] (4) to node{} (2);
		\draw [bend left=270, looseness=2.00] (12) to (13);
		\draw (11.south) to (3.north);
		\draw [arrow, markat=0.66] (8) to (20);
		\draw [arrow, markat=0.33] (0) to (8);
	\end{pgfonlayer}
      \end{tikzpicture}}}
  \]
\item Finally, the cup $\eta_A \colon I \to A^* \tensor A$ in
  $\CPM(\cat{C})$ is given by $(\eta_A)_* \tensor \eta_A = \eta_A
  \tensor \eta_A$ in $\cat{C}$
  (\textit{i.e.}~with ancillary system $I$ and Kraus morphism $\eta_A$ in $\cat{C}$).
\end{itemize}
If $\cat{C}$ is the dagger compact category $\Cat{FHilb}$ of finite-dimensional
Hilbert spaces and linear maps, then $\CPM(\Cat{FHilb})$ is precisely the
category of finite-dimensional Hilbert spaces and completely positive
maps~\cite{selinger:completelypositive}. To come closer to the traditional
setting, we may identify the objects $H$ of $\CPM(\Cat{FHilb})$ with
their algebras of operators $B(H)$. 

The notion of complete positivity makes perfect sense for normal linear
maps between von Neumann algebras $B(H)$ for Hilbert space $H$ of
arbitrary dimension. We now present a $\CP$--construction that works on 
dagger monoidal categories that are not necessarily compact,
and that reduces to the previous construction in the compact
case. Subsequently we prove that applying this construction to
the category of Hilbert spaces indeed results in the traditional
completely positive maps as morphisms.

The idea behind the following construction is to rewrite the morphisms 
\begin{equation}\label{eq:idea}
\begin{aligned}\begin{tikzpicture}[scale=0.75]
	\begin{pgfonlayer}{nodelayer}
		\node [style=none] (0) at (-1, 1) {};
		\node [style=none] (1) at (1, 1) {};
		\node [style=none] (2) at (-1, 0) {};
		\node [style=box] (3) at (-0.75, 0) {$\;\;f_*\;\;$};
		\node [style=none] (4) at (-0.5, 0) {};
		\node [style=none] (5) at (0.5, 0) {};
		\node [style=box] (6) at (0.75, 0) {$\;\;\;f\;\;$};
		\node [style=none] (7) at (1, 0) {};
		\node [style=none] (8) at (-0.75, -1) {};
		\node [style=none] (9) at (0.75, -1) {};
	\end{pgfonlayer}
	\begin{pgfonlayer}{edgelayer}
		\draw [arrow] (3) to (8);
		\draw [arrow, bend right=90, looseness=2.25] (5) to (4);
		\draw [arrow] (0) to (2);
		\draw [arrow, markat=0.66] (9) to (6);
		\draw [arrow, markat=0.66] (7) to (1);
	\end{pgfonlayer}
\end{tikzpicture}\end{aligned}
\qquad\mbox{ as }\qquad
\begin{aligned}\begin{tikzpicture}[scale=0.75]
	\begin{pgfonlayer}{nodelayer}
		\node [style=box] (fd) at (0, 0.85) {$\;\;f^\dag\;\;$};
		\node [style=box] (f) at (0, -0.85) {$\;\;f\phantom{^\dag}\;\;$};
	\end{pgfonlayer}
	\begin{pgfonlayer}{edgelayer}
		\draw [dashed, gray] (.4,0) circle(.35);
		\draw (-.4,-.85) to (-.4,.85);
		\draw (.4,-.85) to (.4,-.35);
		\draw (.4,.85) to (.4,.35);
		\draw (fd) to (0,1.5);
		\draw (f) to (0,-1.5);
	\end{pgfonlayer}
\end{tikzpicture}\end{aligned}
\end{equation}
in a form not using compactness. Composition becomes plugging the hole. For this to be well-defined, we need to take the Kraus morphism $f$ as a representative seriously.

\begin{definition}\label{def:equivalence}
Define an equivalence relation on $\bigcup_{X \in \cat{C}} \cat{C}(A,X \otimes B)$ by setting $f \sim g$ when
  
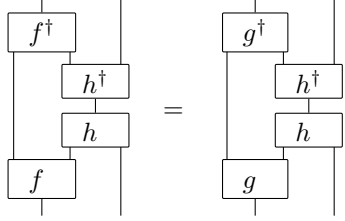
\[\begin{aligned}\begin{tikzpicture}[xscale=0.75,yscale=.65]
	\begin{pgfonlayer}{nodelayer}
		\node [style=box] (fd) at (0, 1.5) {$\;\;f^\dag\;\;$};
		\node [style=box] (hd) at (.95, 0.5) {$\;\;h^\dag\;\;$};
		\node [style=box] (h) at (.95, -0.5) {$\;\;h\phantom{^\dag}\;\;$};
		\node [style=box] (f) at (0, -1.5) {$\;\;f\phantom{^\dag}\;\;$};
	\end{pgfonlayer}
	\begin{pgfonlayer}{edgelayer}
	    \draw (fd) to (0,2.25);
	    \draw (f) to (0,-2.25);
	    \draw (-.5,-1.5) to (-.5,1.5);
	    \draw (h) to (hd);
	    \draw (.5,-1.5) to (.5,-.5);
	    \draw (.5,1.5) to (.5,.5);
	    \draw (1.4,.5) to (1.4,2.25);
	    \draw (1.4,-.5) to (1.4,-2.25);
	\end{pgfonlayer}
  \end{tikzpicture}\end{aligned}
  \quad = \quad 
  \begin{aligned}\begin{tikzpicture}[xscale=0.75,yscale=.65]
	\begin{pgfonlayer}{nodelayer}
		\node [style=box] (fd) at (0, 1.5) {$\;\;g^\dag\;\;$};
		\node [style=box] (hd) at (.95, 0.5) {$\;\;h^\dag\;\;$};
		\node [style=box] (h) at (.95, -0.5) {$\;\;h\phantom{^\dag}\;\;$};
		\node [style=box] (f) at (0, -1.5) {$\;\;g\phantom{^\dag}\;\;$};
	\end{pgfonlayer}
	\begin{pgfonlayer}{edgelayer}
	    \draw (fd) to (0,2.25);
	    \draw (f) to (0,-2.25);
	    \draw (-.5,-1.5) to (-.5,1.5);
	    \draw (h) to (hd);
	    \draw (.5,-1.5) to (.5,-.5);
	    \draw (.5,1.5) to (.5,.5);
	    \draw (1.4,.5) to (1.4,2.25);
	    \draw (1.4,-.5) to (1.4,-2.25);
	\end{pgfonlayer}
  \end{tikzpicture}\end{aligned}\]
  for all $h \colon B \otimes Y \to Z$ in a dagger monoidal category $\cat{C}$.
\end{definition}

\begin{lemma}\label{lem:ancilla}
  If $f \colon A \to X \otimes B$ and $v \colon X \to Y$ in a dagger monoidal category satisfy $v^\dag v = 1$, then $f \sim (v \otimes 1)f$.
\end{lemma}
\begin{proof}
  Immediate from the definition of the equivalence relation.  
\end{proof}

\begin{definition}
  For a dagger monoidal category $\cat{C}$, define a
  new category $\CP(\cat{C})$ as follows.
  \begin{itemize}
  \item The objects of $\CP(\cat{C})$ are those of $\cat{C}$.
  \item The morphisms $A \to B$ of $\CP(\cat{C})$ are equivalence classes of morphisms $f \colon A \to C \otimes B$ of $\cat{C}$.
  \[
    \CP(\cat{C})(A,B) = 
    \{ f \in \cat{C}(A, X \otimes B) \mid X \in \cat{C} \} \,\slash \mathop{\sim}
  \]
  We call a representative $f$ a \emph{Kraus morphism} and $X$ its \emph{ancillary system}.
 \item The coherence isomorphism $A \to I \otimes A$ of $\cat{C}$ represents the identity on $A$ in $\CP(\cat{C})$, and composition is defined as follows.
   \[
   \left[
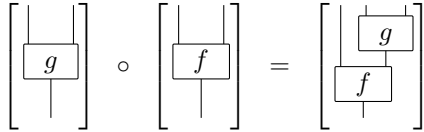
\begin{aligned}\begin{tikzpicture}[scale=0.75]
	\begin{pgfonlayer}{nodelayer}
	     \node [style=box] at (0,0) {$\;\;g\vphantom{f}\;\;$};
	\end{pgfonlayer}
	\begin{pgfonlayer}{edgelayer}
     \draw (0,0) to (0,-1);
     \draw (-.4,0) to (-.4,1);
     \draw (.4,0) to (.4,1);
	\end{pgfonlayer}
   \end{tikzpicture}\end{aligned}\right]
   \;\;\circ\;\;
   \left[\begin{aligned}\begin{tikzpicture}[scale=0.75]
	\begin{pgfonlayer}{nodelayer}
	     \node [style=box] at (0,0) {$\;\;f\;\;$};
	\end{pgfonlayer}
	\begin{pgfonlayer}{edgelayer}
     \draw (0,0) to (0,-1);
     \draw (-.4,0) to (-.4,1);
     \draw (.4,0) to (.4,1);
	\end{pgfonlayer}
   \end{tikzpicture}\end{aligned}\right]
   \;\;=\;\;
   \left[\begin{aligned}\begin{tikzpicture}[scale=0.75]
	\begin{pgfonlayer}{nodelayer}
	     \node [style=box] at (0,-.4) {$\;\;f\;\;$};
	     \node [style=box] at (.4,.5) {$\;\;g\vphantom{f}\;\;$};
	\end{pgfonlayer}
	\begin{pgfonlayer}{edgelayer}
     \draw (0,-.4) to (0,-1);
     \draw (-.4,-.4) to (-.4,1);
     \draw (.4,-.4) to (.4,.5);
     \draw (.05,.5) to (.05,1);
     \draw (.75,.5) to (.75,1);
	\end{pgfonlayer}
   \end{tikzpicture}\end{aligned}\right]
   \]
  \end{itemize}
\end{definition}

\begin{proposition}\label{prop:welldefined}
  If $\cat{C}$ is a dagger monoidal category, $\CP(\cat{C})$ is indeed a well-defined category.
\end{proposition}
\begin{proof}
  Composition is well-defined, because if $f \sim f'$ and $g \sim g'$, then $\alpha^\dag (1 \otimes g) f \sim \alpha^\dag (1 \otimes g') f'$, where we write $\alpha$ for the coherence isomorphism $(X \otimes Y) \otimes C \to X \otimes (Y \otimes C)$. Applying Lemma~\ref{lem:ancilla} to the same coherence isomorphism $\alpha$ shows that composition is associative. Applying Lemma~\ref{lem:ancilla} to the coherence isomorphisms $X \otimes I \to X$ and $I \otimes X \to X$ shows that pre- and post-composition with identities leaves morphisms invariant.
\end{proof}

Before we go on to show that $\CP(\cat{C})$ is a monoidal category, let us discuss size issues. In general, $\CPM(\cat{C})$ is just as large as $\cat{C}$, but this is less clear for $\CP(\cat{C})$.
If $\cat{C}$ is small, then so is $\CP(\cat{C})$, because set-indexed unions of sets are again sets~\cite{halmos:settheory}. 
If $\cat{C}$ is essentially small, meaning that it is dagger equivalent to a small category, then so is $\CP(\cat{C})$, by Lemma~\ref{lem:ancilla}. 
If $\cat{C}$ is locally small, it is not clear that $\CP(\cat{C})$ is again locally small. This does hold for large but locally small categories $\cat{C}$ such as $\Cat{Hilb}$ and $\Cat{Rel}$, that are well-powered and have good notions of image~\cite{heunenjacobs:daggerkernel} and rank: if a morphism $f \colon A \to X \otimes B$ preserves rank, its image can at most be as large as $A$, so we may restrict to ancillary systems $X$ that are subobjects of $A$, which constitute a set. 
Compare this to the existence of a minimal Stinespring dilation, or the fact that it is enough to check $n$-positivity to verify that a map between $n$-dimensional systems is completely positive by Choi's theorem~\cite{paulsen:completelypositive}. 

\begin{proposition}\label{prop:tensorinCP}
  If $\cat{C}$ is a dagger braided monoidal category, then $\CP(\cat{C})$ is a braided monoidal category.
  The tensor unit $I$ and tensor products of objects are as in $\cat{C}$, 
  and the tensor product of morphisms is defined as follows.
  \[
   \left[\begin{aligned}\begin{tikzpicture}[scale=0.75]
	\begin{pgfonlayer}{nodelayer}
	     \node [style=box] at (0,0) {$\;\;f\;\;$};
	\end{pgfonlayer}
	\begin{pgfonlayer}{edgelayer}
     \draw (0,0) to (0,-.8);
     \draw (-.4,0) to (-.4,1);
     \draw (.4,0) to (.4,1);
	\end{pgfonlayer}
   \end{tikzpicture}\end{aligned}\right]
   \;\;\otimes\;\;
   \left[\begin{aligned}\begin{tikzpicture}[scale=0.75]
	\begin{pgfonlayer}{nodelayer}
	     \node [style=box] at (0,0) {$\;\;g\vphantom{f}\;\;$};
	\end{pgfonlayer}
	\begin{pgfonlayer}{edgelayer}
     \draw (0,0) to (0,-.8);
     \draw (-.4,0) to (-.4,1);
     \draw (.4,0) to (.4,1);
	\end{pgfonlayer}
   \end{tikzpicture}\end{aligned}\right]
   \;\;=\;\;
   \left[\begin{aligned}\begin{tikzpicture}[scale=0.75]
	\begin{pgfonlayer}{nodelayer}
	     \node [style=box] at (0,0) {$\;\;f\;\;$};
	     \node [style=box] at (1.5,0) {$\;\;g\vphantom{f}\;\;$};
	\end{pgfonlayer}
	\begin{pgfonlayer}{edgelayer}
     \draw (0,0) to (0,-.8);
     \draw (1.5,0) to (1.5,-.8);
     \draw (-.4,0) to (-.4,1);
     \draw (1.9,0) to (1.9,1);
     \draw (1.1,.2) to [out=up,in=down] (.4,1);
     \draw[cross] (.4,.2) to [out=up,in=down] (1.1,1);
	\end{pgfonlayer}
   \end{tikzpicture}\end{aligned}\right]
  \]
   Moreover, if $\cat{C}$ is dagger symmetric monoidal, then $\CP(\cat{C})$ is symmetric monoidal.
\end{proposition}
\begin{proof}
  Write $\alpha_{A,B,C}\colon (A \otimes B) \otimes C \to A \otimes (B \otimes C)$, $\lambda_A \colon I \otimes A \to A$, and $\rho_A \colon A \otimes I \to A$ for the coherence isomorphisms of $\cat{C}$, and $\sigma_{A,B} \colon A \otimes B \to B \otimes A$ for the braiding. The coherence isomorphisms in $\CP(\cat{C})$ are 
  \begin{align*}
    \alpha_{A,B,C} & = [ \lambda_{A \otimes (B \otimes C)} \alpha_{A,B,C} ], &
    \lambda_A & = [ 1_{I \otimes A} ], &
    \rho_A & = [ \sigma_{A,I} ], &
    \sigma_{A,B} & = [ \lambda_{B \otimes A} \sigma_{A,B} ].
  \end{align*}
  Verifying the triangle, pentagon, and hexagon laws is laborious but straightforward, and repeatedly uses Lemma~\ref{lem:ancilla} and $\lambda_I=\rho_I$~\cite[1.1]{joyalstreet:braidedtensorcategories}. Similarly, $\CP(\cat{C})$ inherits symmetry from $\cat{C}$: 
  \begin{align*}
    \sigma_{A,B} \sigma_{B,A} 
    & = [ \lambda_{B \otimes A} \sigma_{A,B}] \circ [\lambda_{A \otimes B} \sigma_{B,A}] \\
    & = [ (1_{I \otimes I} \otimes \sigma_{B,A}) \alpha^\dag_{I,I,B \otimes A} \lambda_{I \otimes (B \otimes A)} (1_I \otimes \sigma_{A,B}) \lambda_{A \otimes B} ] \\
    & = [ \alpha^\dag_{I,I,A \otimes B} \lambda_{I \otimes (A \otimes B)} \lambda_{A \otimes B} ] \\
    & = [ \lambda_{A \otimes B} ] \\
    & = 1_{A \otimes B},
  \end{align*}
  where the third equation follows from $\sigma_{A,B}\sigma_{B,A}=1_{A \otimes B}$ in $\cat{C}$, and the fourth equation follows from an application of Lemma~\ref{lem:ancilla} to $\rho_I \colon I \otimes I \to I$.
\end{proof}

Notice that $\CP(\cat{C})$ does not obviously have a dagger, even though $\cat{C}$ does. 
Together with the fact that $\CP(\cat{C})$ is not obviously locally small when $\cat{C}$ is, this suggests that the \CP--construction is not straightforwardly a monad on dagger braided monoidal categories.

\begin{proposition}\label{prop:generalizedCP}
  If $\cat{C}$ is a dagger compact category, then $\CPM(\cat{C})$ and
  $\CP(\cat{C})$ are isomorphic dagger symmetric monoidal categories. 
  Hence if $\cat{C}$ is a dagger compact category, then so is $\CP(\cat{C})$.
\end{proposition}
\begin{proof}
  The isomorphism $\CPM(\cat{C}) \to \CP(\cat{C})$ is given
  by $A \mapsto A$ on objects, and acts as follows on morphisms.
  
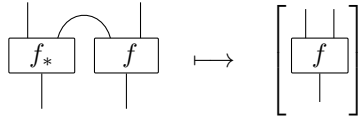
\[
    \begin{aligned}\begin{tikzpicture}[scale=0.75]
	\begin{pgfonlayer}{nodelayer}
		\node [style=none] (0) at (-1, 1) {};
		\node [style=none] (1) at (1, 1) {};
		\node [style=none] (2) at (-1, 0) {};
		\node [style=box] (3) at (-0.75, 0) {$\;\;f_*\;\;$};
		\node [style=none] (4) at (-0.5, 0) {};
		\node [style=none] (5) at (0.5, 0) {};
		\node [style=box] (6) at (0.75, 0) {$\;\;\;f\;\;$};
		\node [style=none] (7) at (1, 0) {};
		\node [style=none] (8) at (-0.75, -1) {};
		\node [style=none] (9) at (0.75, -1) {};
	\end{pgfonlayer}
	\begin{pgfonlayer}{edgelayer}
		\draw (3) to (8);
		\draw [bend right=90, looseness=2.25] (5) to (4);
		\draw (2) to (0);
		\draw (6) to (9);
		\draw (7) to (1);
	\end{pgfonlayer}
    \end{tikzpicture}\end{aligned}
    \quad\longmapsto\quad
    \left[\begin{aligned}\begin{tikzpicture}[xscale=0.75,yscale=.66]
	\begin{pgfonlayer}{nodelayer}
		\node [style=none] (0) at (-1, 1) {};
		\node [style=none] (1) at (-0.5, 1) {};
		\node [style=none] (2) at (-1, 0) {};
		\node [style=box] (3) at (-0.75, 0) {$\;\;f\;\;$};
		\node [style=none] (4) at (-0.5, 0) {};
		\node [style=none] (5) at (-0.75, -1) {};
	\end{pgfonlayer}
	\begin{pgfonlayer}{edgelayer}
		\draw (5) to (3);
		\draw (4) to (1);
		\draw (2) to (0);
	\end{pgfonlayer}
    \end{tikzpicture}\end{aligned}\right]
  \]
  This assignment is clearly invertible. To see that it gives a
  well-defined functor preserving daggers and symmetric monoidal
  structure takes light work. 
  See~\cite[Lemma~2.2]{coecke:mix} for details.
\end{proof}

If canonical Kraus morphisms are available, as in $\cat{Hilb}$ and $\cat{Rel}$, we may use an alternative description of $\CP(\cat{C})$ that does not use equivalence relations, as in~\cite{coeckeheunen:qpl}. The following corollary spells this out for compact $\cat{C}$, such as $\cat{Rel}$.

\begin{corollary}\label{cor:CPwithouteqrel}
  In a dagger compact category, $f \sim g$ if and only if
  
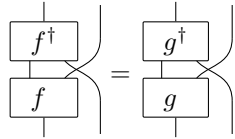
\[
    \begin{aligned}\begin{tikzpicture}[scale=0.75]
	\begin{pgfonlayer}{nodelayer}
		\node [style=none] (0) at (1, 1.2) {};
		\node [style=none] (1) at (0, 1.2) {};
		\node [style=none] (2) at (1, 0.5) {};
		\node [style=none] (3) at (0.25, 0.5) {};
		\node [style=box] (4) at (0, 0.5) {$\;\;f^\dag\;\;$};
		\node [style=none] (5) at (-0.25, 0.5) {};
		\node [style=none] (6) at (1, -0.5) {};
		\node [style=none] (7) at (0.25, -0.5) {};
		\node [style=box] (8) at (0, -0.5) {$\;\;f\phantom{^\dag}\;\;$};
		\node [style=none] (9) at (-0.25, -0.5) {};
		\node [style=none] (10) at (1, -1.2) {};
		\node [style=none] (11) at (0, -1.2) {};
	\end{pgfonlayer}
	\begin{pgfonlayer}{edgelayer}
		\draw (9.center) to (5.center);
		\draw [in=90, out=270, looseness=1.25] (3.center) to (6.south);
		\draw (10) to (6);
		\draw (4) to (1.center);
		\draw [in=90, out=-90, looseness=1.25] (2.north) to (7.center);
		\draw (2) to (0);
		\draw (8) to (11.center);
	\end{pgfonlayer}\end{tikzpicture}\end{aligned}
	=
    \begin{aligned}\begin{tikzpicture}[scale=0.75]
	\begin{pgfonlayer}{nodelayer}
		\node [style=none] (0) at (1, 1.2) {};
		\node [style=none] (1) at (0, 1.2) {};
		\node [style=none] (2) at (1, 0.5) {};
		\node [style=none] (3) at (0.25, 0.5) {};
		\node [style=box] (4) at (0, 0.5) {$\;\;g^\dag\;\;$};
		\node [style=none] (5) at (-0.25, 0.5) {};
		\node [style=none] (6) at (1, -0.5) {};
		\node [style=none] (7) at (0.25, -0.5) {};
		\node [style=box] (8) at (0, -0.5) {$\;\;g\phantom{^\dag}\;\;$};
		\node [style=none] (9) at (-0.25, -0.5) {};
		\node [style=none] (10) at (1, -1.2) {};
		\node [style=none] (11) at (0, -1.2) {};
	\end{pgfonlayer}
	\begin{pgfonlayer}{edgelayer}
		\draw (9.center) to (5.center);
		\draw [in=90, out=270, looseness=1.25] (3.center) to (6.south);
		\draw (10) to (6);
		\draw (4) to (1.center);
		\draw [in=90, out=-90, looseness=1.25] (2.north) to (7.center);
		\draw (2) to (0);
		\draw (8) to (11.center);
	\end{pgfonlayer}\end{tikzpicture}\end{aligned}
  \]
  for $f \colon A \to X \otimes B$ and $g \colon A \to Y \otimes B$.
\end{corollary}
\begin{proof}
  By definition, $f \sim g$ in $\cat{C}$ means precisely that $[f]=[g]$ in $\CP(\cat{C})$.
  Via the isomorphism of Proposition~\ref{prop:generalizedCP}, elementary graphical manipulation then yields the statement.
\end{proof}

\begin{remark}\label{remark:composition}
  One might try, as in~\cite{coeckeheunen:qpl}, to define $\CP(\cat{C})$ as a monoidal subcategory of $\cat{C}$ as follows.
    \begin{align*}
      \CP(\cat{C})(A,B) \;&=\; \left\{\left.
          \vcenter{\hbox{\begin{tikzpicture}[scale=0.75]
	\begin{pgfonlayer}{nodelayer}
		\node [style=none] (0) at (1, 1.2) {};
		\node [style=none] (1) at (0, 1.2) {};
		\node [style=none] (2) at (1, 0.5) {};
		\node [style=none] (3) at (0.25, 0.5) {};
		\node [style=box] (4) at (0, 0.5) {$\;\;f^\dag\;\;$};
		\node [style=none] (5) at (-0.25, 0.5) {};
		\node [style=none] (6) at (1, -0.5) {};
		\node [style=none] (7) at (0.25, -0.5) {};
		\node [style=box] (8) at (0, -0.5) {$\;\;f\phantom{^\dag}\;\;$};
		\node [style=none] (9) at (-0.25, -0.5) {};
		\node [style=none] (10) at (1, -1.2) {};
		\node [style=none] (11) at (0, -1.2) {};
	\end{pgfonlayer}
	\begin{pgfonlayer}{edgelayer}
		\draw (9.center) to (5.center);
		\draw [in=90, out=270, looseness=1.25] (3.center) to (6.south);
		\draw (10) to (6);
		\draw (4) to (1.center);
		\draw [cross, in=90, out=-90, looseness=1.25] (2.north) to (7.center);
		\draw (2) to (0);
		\draw (8) to (11.center);
	\end{pgfonlayer}
      \end{tikzpicture}}}
      \;\;\right|\;\;
      \vcenter{\hbox{\begin{tikzpicture}[scale=0.75]
	\begin{pgfonlayer}{nodelayer}
		\node [style=none] (0) at (-1, 1) {};
		\node [style=none] (1) at (-0.5, 1) {};
		\node [style=none] (2) at (-1, 0) {};
		\node [style=box] (3) at (-0.75, 0) {$\;\;f\;\;$};
		\node [style=none] (4) at (-0.5, 0) {};
		\node [style=none] (5) at (-0.75, -1) {};
	\end{pgfonlayer}
	\begin{pgfonlayer}{edgelayer}
		\draw (3) to (5);
		\draw (1) to (4);
		\draw (2) to (0);
	\end{pgfonlayer}
      \end{tikzpicture}}}
      \in \cat{C}(A,C \tensor B) \right\}
   \\
    \left(
          \vcenter{\hbox{\begin{tikzpicture}[scale=0.75]
	\begin{pgfonlayer}{nodelayer}
		\node [style=none] (0) at (1, 1.2) {};
		\node [style=none] (1) at (0, 1.2) {};
		\node [style=none] (2) at (1, 0.5) {};
		\node [style=none] (3) at (0.25, 0.5) {};
		\node [style=box] (4) at (0, 0.5) {$\;\;g^\dag\;\;$};
		\node [style=none] (5) at (-0.25, 0.5) {};
		\node [style=none] (6) at (1, -0.5) {};
		\node [style=none] (7) at (0.25, -0.5) {};
		\node [style=box] (8) at (0, -0.5) {$\;\;g\phantom{^\dag}\;\;$};
		\node [style=none] (9) at (-0.25, -0.5) {};
		\node [style=none] (10) at (1, -1.2) {};
		\node [style=none] (11) at (0, -1.2) {};
	\end{pgfonlayer}
	\begin{pgfonlayer}{edgelayer}
		\draw (9.center) to (5.center);
		\draw [in=90, out=270, looseness=1.25] (3.center) to (6.south);
		\draw (10) to (6);
		\draw (4) to (1.center);
		\draw [cross, in=90, out=-90, looseness=1.25] (2.north) to (7.center);
		\draw (2) to (0);
		\draw (8) to (11.center);
	\end{pgfonlayer}
      \end{tikzpicture}}}
    \right)
    \;\after\;
    \left(
          \vcenter{\hbox{\begin{tikzpicture}[scale=0.75]
	\begin{pgfonlayer}{nodelayer}
		\node [style=none] (0) at (1, 1.2) {};
		\node [style=none] (1) at (0, 1.2) {};
		\node [style=none] (2) at (1, 0.5) {};
		\node [style=none] (3) at (0.25, 0.5) {};
		\node [style=box] (4) at (0, 0.5) {$\;\;f^\dag\;\;$};
		\node [style=none] (5) at (-0.25, 0.5) {};
		\node [style=none] (6) at (1, -0.5) {};
		\node [style=none] (7) at (0.25, -0.5) {};
		\node [style=box] (8) at (0, -0.5) {$\;\;f\phantom{^\dag}\;\;$};
		\node [style=none] (9) at (-0.25, -0.5) {};
		\node [style=none] (10) at (1, -1.2) {};
		\node [style=none] (11) at (0, -1.2) {};
	\end{pgfonlayer}
	\begin{pgfonlayer}{edgelayer}
		\draw (9.center) to (5.center);
		\draw [in=90, out=270, looseness=1.25] (3.center) to (6.south);
		\draw (10) to (6);
		\draw (4) to (1.center);
		\draw [cross, in=90, out=-90, looseness=1.25] (2.north) to (7.center);
		\draw (2) to (0);
		\draw (8) to (11.center);
	\end{pgfonlayer}
      \end{tikzpicture}}}
    \right)
    \;\;&=\;\;\vcenter{\hbox{
        \begin{tikzpicture}[scale=0.75]
	\begin{pgfonlayer}{nodelayer}
		\node [style=none] (0) at (1, 2.5) {};
		\node [style=none] (1) at (0, 2.5) {};
		\node [style=None] (2) at (0.75, 2.25) {};
		\node [style=None] (3) at (-1.25, 2.25) {};
		\node [style=none] (4) at (0.25, 1.75) {};
		\node [style=none] (5) at (-0.5, 1.75) {};
		\node [style=box] (6) at (0, 1.75) {$\;\;f^\dag\;\;$};
		\node [style=none] (7) at (0.25, 0.75) {};
		\node [style=box] (8) at (0, 0.75) {$\;\;g^\dag\;\;$};
		\node [style=none] (9) at (-0.25, 0.75) {};
		\node [style=none] (10) at (-1, 0.75) {};
		\node [style=none] (11) at (1, 0.5) {};
		\node [style=None] (12) at (0.75, 0.25) {};
		\node [style=None] (13) at (-1.25, 0.25) {};
		\node [style=None] (14) at (0.75, -0.25) {};
		\node [style=None] (15) at (-1.25, -0.25) {};
		\node [style=none] (16) at (1, -0.5) {};
		\node [style=none] (17) at (0.25, -0.75) {};
		\node [style=box] (18) at (0, -0.75) {$\;\;g\phantom{^\dag}\;\;$};
		\node [style=none] (19) at (-0.25, -0.75) {};
		\node [style=none] (20) at (-1, -0.75) {};
		\node [style=none] (21) at (0.25, -1.75) {};
		\node [style=none] (22) at (-0.5, -1.75) {};
		\node [style=box] (23) at (0, -1.75) {$\;\;f\phantom{^\dag}\;\;$};
		\node [style=None] (24) at (0.75, -2.25) {};
		\node [style=None] (25) at (-1.25, -2.25) {};
		\node [style=none] (26) at (1, -2.5) {};
		\node [style=none] (27) at (0, -2.5) {};
	\end{pgfonlayer}
	\begin{pgfonlayer}{edgelayer}
		\draw [dashed, gray] (13) to node{} (12);
		\draw [in=307, out=90] (10.south) to (5);
		\draw [in=270, out=90, looseness=1.25] (16.south) to (7);
		\draw (1) to (6);
		\draw (26) to (16);
		\draw (10) to (20);
		\draw [in=53, out=-90] (20.north) to (22);
		\draw [dashed, gray] (25) to node{} (24);
		\draw (17) to (21);
		\draw [cross, in=90, out=-90, looseness=1.25] (11.north) to (17);
		\draw (4) to (7);
		\draw (0) to (11);
		\draw [dashed, gray] (14) to node{} (15);
		\draw (23) to (27);
		\draw [dashed, gray] (15) to node{} (25);
		\draw [dashed, gray] (12) to node{} (2);
		\draw [dashed, gray] (2) to node{} (3);
		\draw [dashed, gray] (24) to node{} (14);
		\draw (9) to (19);
		\draw [dashed, gray] (3) to node{} (13);
	\end{pgfonlayer}
      \end{tikzpicture}    }}
   \end{align*}
   That is, the homset $\CP(\cat{C})(A,B)$ is a subset of $\cat{C}(A \otimes B, A \otimes B)$, and the equivalence relation $\sim$ is simply absorbed by the set comprehension notation.
   However, it is unclear that this composition is well-defined.
   Clearly the composition of morphisms $F \colon A \to B$ and $G \colon B \to C$ does not depend on the chosen representative $g \colon B \to Y \otimes C$ in $G=(g^\dag \otimes \id) (\id \otimes \sigma) (g \otimes \id)$. But the composition does in general depend on the choice of representative of $F$.

   The category $\cat{Hilb}$ has two properties that make this composition well-defined nevertheless. First, it is \emph{monoidally well-pointed}: if $f (a \times b)= g(a \otimes b)$ for all $a \colon I \to A$ and $b \colon I \to B$, then $f=g \colon A \otimes B \to C$~\cite{abramskyheunen:hstar}. Second, it is \emph{enriched} over Banach spaces: by choosing orthonormal bases, any morphism $f \colon A \to B$ can be written in `matrix form' as $f=\sum_i b_i a_i^\dag$ for $a_i \colon I \to A$ and $b_i \colon I \to B$, and this (possibly infinite) sum respects composition and tensor products.
   Using these properties one sees that $(\id \otimes b^\dag)GF(\id \otimes a)$ is independent of the chosen representative of $F$, and hence so is $GF$.
\end{remark}

We now embark on showing that the $\CP$--construction is not just formal manipulation of diagrams that just happens to coincide with the traditional setting in the case of finite-dimensional Hilbert spaces.

To do so, we first recall the definition of quantum operations in the
Heisenberg picture, as it is usually stated in infinite-dimensional
quantum information theory~\cite{holevo:statistical}. A function
$\varphi \colon A \to B$ between von Neumann algebras is \emph{unital}
when $\varphi(1)=1$, and \emph{positive} when for each $a \in A$ there is
$b \in B$ such that $\varphi(a^*a)=b^*b$. When it is linear, the
function $\varphi$ is called \emph{normal} when it preserves suprema
of increasing nets of projections, or equivalently, when it is
ultraweakly continuous. Finally,
$\varphi$ is \emph{completely positive} when $\varphi \tensor \id
\colon A \tensor M_n \to B \tensor M_n$ is positive for all $n \in
\mathbb{N}$, where $M_n$ is the von Neumann algebra of $n$-by-$n$
complex matrices. (See e.g.~\cite[3.3]{takesaki:operatoralgebras}
or~\cite[p26]{paulsen:completelypositive}.)  

\begin{definition}\label{def:cp}
  A \emph{quantum operation} is a normal completely positive linear
  map between von Neumann algebras.\footnote{Quantum operations $\varphi$ are usually also taken to be \emph{subunital}, that is, satisfying $\varphi(\id[H]) \leq \id[K]$. We do not require this for the same reason as the original CPM--construction did not~\cite[Remark~6.1]{selinger:dataflow}. See also the discussion after Theorem~\ref{thm:justification}.}\footnote{Not every completely
    positive linear map between von Neumann algebras is normal. In
    fact, a positive linear map is normal if and only if it is weak-*
    continuous~\cite[46.5]{conway:operatortheory}.} 
  Hilbert
  spaces $H$ and quantum operations $B(H) \to B(K)$ 
  form a category that we denote by $\Cat{QOperations}$. 
\end{definition}

To see that $\CP(\Cat{Hilb})$ is isomorphic to $\Cat{QOperations}$, we
will rely on two classical theorems: Stinespring's dilation theorem,
showing that quantum operations can be written as *-homomorphisms on
larger algebras, and Dixmier's structure theorem for normal
*-homomorphisms. A \emph{*-homomorphism} is a
linear map $\pi \colon A \to B$ between C*-algebras that satisfies
$\pi(ab)=\pi(a)\pi(b)$ and $\pi(a^*)=\pi(a)^*$.  

Let us emphasize that the following results hold for arbitrary Hilbert
spaces, not just separable ones.

\begin{theorem}[Stinespring]\label{thm:stinespring}
  Let $A$ be a von Neumann algebra. For any normal completely positive linear
  map $\varphi \colon A \to B(H)$, there exist a unital normal
  *-homomorphism $\pi \colon A \to B(K)$ and a continuous linear $v
  \colon H \to K$ such that $\varphi(a)=v^\dag \pi(a) v$.  
\end{theorem}
\begin{proof}
  See~\cite[Theorem~III.2.2.4]{blackadar:operatoralgebra}.
\end{proof}

\begin{theorem}[Dixmier]\label{thm:dixmier}
  Every normal *-homomorphism $\varphi \colon B(H) \to B(K)$
  factors as $\varphi = \varphi_3 \varphi_2 \varphi_1$ for:
  \begin{itemize}
  \item an \emph{ampliation} $\varphi_1 \colon B(H) \to B(H \tensor H') \colon f \mapsto f
    \tensor \id[H']$ for some Hilbert space $H'$;
  \item an \emph{induction} $\varphi_2 \colon B(H \tensor H') \to B(K') \colon f \mapsto p
     f p$ for the projection $p \in B(H \tensor H')$ onto some Hilbert subspace
    $K' \subseteq H \tensor H'$;
  \item a \emph{spatial isomorphism} $\varphi_3 \colon B(K') \to B(K) \colon f \mapsto
    u^\dag f u$ for a unitary $u \colon K \to K'$.
  \end{itemize}
\end{theorem}
\begin{proof}
  See~\cite[I.4.3]{dixmier:vonneumann}
  or~\cite[IV.5.5]{takesaki:operatoralgebras}.
\end{proof}

\begin{corollary}\label{cor:quantumchannel}
  A linear map $\varphi \colon B(H) \to B(K)$ is a quantum operation
  if and only if it is of the form $\varphi(f)=g^\dag  (f
  \tensor \id) g$ for some Hilbert space $H'$ and continuous
  linear map $g \colon K \to H \tensor H'$. 
\end{corollary}
\begin{proof}
  Given a quantum operation $\varphi$, combine Theorems~\ref{thm:stinespring}
  and~\ref{thm:dixmier} to get $\varphi(f)=v^\dag u^\dag p (f \tensor \id)
  p u v$; taking $g=p u v$ brings $\varphi$ into the required
  form. 

  Conversely, one easily checks directly via Definition~\ref{def:cp} that a map $\varphi$ of
  the given form is indeed normal, and completely positive.
\end{proof}

\begin{theorem}\label{thm:justification}
  $\CP(\Cat{Hilb})$ and $\Cat{QOperations}$ are contravariantly isomorphic symmetric monoidal categories.
\end{theorem}
\begin{proof}
  The objects are already equal. The isomorphism sends a
  morphism $H \to K$ of $\CP(\Cat{Hilb})$ with Kraus morphism $g\colon H \to X \otimes K$ to the
  quantum operation $g^\dag (\id \tensor \blank) g \colon B(K) \to B(H)$. 

  This prescription is clearly (contravariantly) functorial and preserves tensor products.
  We show that it is well-defined. Suppose that $g \sim h$. Then by Definition~\ref{def:equivalence}, $g^\dag(\id \tensor f^\dag f)g=h^\dag(\id \tensor f^\dag f)h$; that is, the functions $g^\dag(\id \tensor \blank)g$ and $h^\dag(\id \tensor\blank)h$ are equal on positive elements $f^\dag f$ of $B(K)$. But then these linear maps coincide on all of $B(K)$ because that C*-algebra is spanned by its positive elements.

  Finally, this monoidal functor is invertible because of Theorem~\ref{thm:stinespring}.
\end{proof}

We end this section by discussing \emph{quantum channels}, that are usually defined as unital quantum operations; preserving units is the equivalent in the Schr{\"o}dinger picture of preserving traces in the Heisenberg picture. If $\varphi(1)=1$ in Stinespring's Theorem~\ref{thm:stinespring}, then $v^\dag v=1$, \textit{i.e.}\ $v$ is an isometry. This holds in the abstract, too.
A morphism in $\CPM(\cat{C})$ has this property when
\[\begin{aligned}\begin{tikzpicture}[scale=0.75]
	\begin{pgfonlayer}{nodelayer}
		\node [style=none] (0) at (-1, 1) {};
		\node [style=none] (1) at (1, 1) {};
		\node [style=none] (2) at (-1, 0) {};
		\node [style=box] (3) at (-0.75, 0) {$\;\;f_*\;\;$};
		\node [style=none] (4) at (-0.5, 0) {};
		\node [style=none] (5) at (0.5, 0) {e};
		\node [style=box] (6) at (0.75, 0) {$\;\;\;f\;\;$};
		\node [style=none] (7) at (1, 0) {};
		\node [style=none] (8) at (-0.75, -1) {};
		\node [style=none] (9) at (0.75, -1) {};
	\end{pgfonlayer}
	\begin{pgfonlayer}{edgelayer}
		\draw [arrow] (3) to (8);
		\draw [arrow, bend right=90, looseness=2.25] (5) to (4);
		\draw [arrow, markat=0.66] (9) to (6);
		\draw [arrow, bend right=90, looseness=2, markat=0.5] (7) to (2);
	\end{pgfonlayer}
  \end{tikzpicture}\end{aligned}
  \quad = \quad
  \begin{aligned}\begin{tikzpicture}[scale=0.75]
	\begin{pgfonlayer}{nodelayer}
		\node [style=none] (8) at (-0.75, -1) {};
		\node [style=none] (9) at (0.75, -1) {};
	\end{pgfonlayer}
	\begin{pgfonlayer}{edgelayer}
		\draw [arrow, bend right=90, looseness=2.25] (9) to (8);
	\end{pgfonlayer}
  \end{tikzpicture}\end{aligned}
\]
Via the analogy of equation~\eqref{eq:idea}, this means for morphisms in $\CP(\cat{C})$ that \emph{any} Kraus morphism is an isometry. This is indeed independent of the Kraus morphism, for if $f \sim g$ for $f \colon A \to X \otimes B$ and $g \colon A \to Y \otimes B$, then $f^\dag f = g^\dag g$ by applying Definition~\ref{def:equivalence} to $h=\rho_B\colon B \otimes I \to B$.

 


\section{Environment structures}
\label{sec:environment}

This section investigates when a given category is of the form
$\CP(\cat{C})$ by generalizing~\cite{coecke:selinger,coeckeperdrix:channels}.

\begin{definition}\label{def:environment}
  An \emph{environment structure} for a dagger monoidal category $\cat{C}$ is a monoidal category $\cat{\widehat{C}}$ with the same objects, together with a strict monoidal functor $F \colon \cat{C} \to \cat{\widehat{C}}$ with $F(A)=A$, and for each object $A$ a morphism $\top_A \colon A \to I$ in $\cat{\widehat{C}}$, depicted as
  $\vcenter{\hbox{\begin{tikzpicture}[scale=0.5]
	\begin{pgfonlayer}{nodelayer}
                \ground{0}{0,0.75}
		\node [style=none] (1) at (0, 0) {};
	\end{pgfonlayer}
	\begin{pgfonlayer}{edgelayer}
		\draw (1) to (0);
	\end{pgfonlayer}
  \end{tikzpicture}}}$, satisfying: 
  \begin{enumerate}
  \item[(a)] We have $\top_I=\id[I]$, and for all objects $A$ and $B$: $\quad
    \vcenter{\hbox{\begin{tikzpicture}[scale=0.75]
	\begin{pgfonlayer}{nodelayer}
                \ground{0}{-0.4,0.75}
		\node [style=none] (1) at (-0.4, 0) {$A$};
                \ground{2}{0.4,0.75}
		\node [style=none] (3) at (0.4, 0) {$B$};
	\end{pgfonlayer}
	\begin{pgfonlayer}{edgelayer}
                \draw (1) to (0);
                \draw (3) to (2);
	\end{pgfonlayer}
      \end{tikzpicture}}} 
   \; = \;
      \vcenter{\hbox{\begin{tikzpicture}[scale=0.75]
	\begin{pgfonlayer}{nodelayer}
                \ground{0}{0,0.75}
		\node [style=none] (1) at (0, 0) {$A \tensor B$};
	\end{pgfonlayer}
	\begin{pgfonlayer}{edgelayer}
		\draw (1) to (0);
	\end{pgfonlayer}
      \end{tikzpicture}}}
  \;\;$ in $\widehat{\cat{C}}$;
\item[(b)] 
	Maps $f \colon A \to X \otimes B$ and $g \colon A \to Y \otimes B$ in $\cat{C}$ are equivalent if and only if
      $\vcenter{\hbox{\begin{tikzpicture}[scale=0.75]
	\begin{pgfonlayer}{nodelayer}
		\node [style=none] (0) at (0.25, 1) {};
		\ground{1}{-0.25, 0.75};
		\node [style=none] (2) at (0.25, 0) {};
		\node [style=box] (3) at (0, 0) {$\;\;Ff\;\;$};
		\node [style=none] (4) at (-0.25, 0) {};
		\node [style=none] (5) at (0, -0.75) {};
	\end{pgfonlayer}
	\begin{pgfonlayer}{edgelayer}
		\draw (5.center) to (3);
		\draw (2.center) to (0.center);
		\draw (1) to (4.center);
	\end{pgfonlayer}
      \end{tikzpicture}}} \;=\;
      \vcenter{\hbox{\begin{tikzpicture}[scale=0.75]
	\begin{pgfonlayer}{nodelayer}
		\node [style=none] (0) at (0.25, 1) {};
		\ground{1}{-0.25, 0.75};
		\node [style=none] (2) at (0.25, 0) {};
		\node [style=box] (3) at (0, 0) {$\;\;Fg\vphantom{f}\;\;$};
		\node [style=none] (4) at (-0.25, 0) {};
		\node [style=none] (5) at (0, -0.75) {};
	\end{pgfonlayer}
	\begin{pgfonlayer}{edgelayer}
		\draw (5.center) to (3);
		\draw (2.center) to (0.center);
		\draw (1) to (4.center);
	\end{pgfonlayer}
      \end{tikzpicture}}}
 \;\;$ in $\widehat{\cat{C}}$;
\item[(c)] For each $\widehat{f} \in \cat{\widehat{C}}(A,B)$ there is
  $f \in \cat{C}(A,X \tensor B)$ such that $\;\;
  \vcenter{\hbox{\begin{tikzpicture}[scale=0.75]
	\begin{pgfonlayer}{nodelayer}
		\node [style=none] (0) at (0, 1) {};
	    \node [style=box] (1) at (0, 0.125) {$\;\;\widehat{f}\;\;$};
		\node [style=none] (2) at (0, -0.75) {};
	\end{pgfonlayer}
	\begin{pgfonlayer}{edgelayer}
		\draw (0.center) to (1);
		\draw (2.center) to (1);
	\end{pgfonlayer}
      \end{tikzpicture}}}
      \; = \;
      \vcenter{\hbox{\begin{tikzpicture}[scale=0.75]
	\begin{pgfonlayer}{nodelayer}
		\node [style=none] (0) at (0.25, 1) {};
		\ground{1}{-0.25, 0.75};
		\node [style=none] (2) at (0.25, 0) {};
		\node [style=box] (3) at (0, 0) {$\;\;Ff\;\;$};
		\node [style=none] (4) at (-0.25, 0) {};
		\node [style=none] (5) at (0, -0.75) {};
	\end{pgfonlayer}
	\begin{pgfonlayer}{edgelayer}
		\draw (5.center) to (3);
		\draw (2.center) to (0.center);
		\draw (1) to (4.center);
	\end{pgfonlayer}
      \end{tikzpicture}}} 
  \;\;$ in $\widehat{\cat{C}}$. 
  \end{enumerate}
  If the functor $F$ is faithful, we call the environment structure \emph{faithful}.
\end{definition}

As Theorem~\ref{thm:environmentnecessary} below makes precise, the category $\cat{Hilb}$ has an environment structure, but the functor $F \colon \cat{Hilb} \to \CP(\cat{Hilb})$ identifies global phases and is therefore not faithful.

Intuitively, if we think of the category $\cat{C}$ as consisting of
pure states, then the supercategory $\cat{\widehat{C}}$ consists of
mixed states. The maps $\top$ `ground' a system within the
environment. Condition (c) then reads that every mixed state can be
seen as a pure state in an extended system; the lack of knowledge
carried in the ancillary system represents the variables
relative to which we mix. 

Having an environment structure is a sufficient condition for $\cat{\widehat{C}}$ to be of the form $\CP(\cat{C})$.

\begin{theorem}\label{thm:environmentCP}
  If a dagger braided monoidal category $\cat{C}$ comes with an
  environment structure, then there exists an isomorphism 
  $\CP(\cat{C}) \to \cat{\widehat{C}}$ of monoidal categories.
\end{theorem}
\begin{proof}
  Define $\xi \colon \CP(\cat{C}) \to \cat{\widehat{C}}$ by setting $\xi(A)=A$ on objects, and
  \[
    \xi \left( \left[
    \begin{aligned}\begin{tikzpicture}[scale=0.5]
	\begin{pgfonlayer}{nodelayer}
	  \node [box] (f) at (0,0) {$\;\;f\;\;$};
	\end{pgfonlayer}
	\begin{pgfonlayer}{edgelayer}
	  \draw (-.4,0) to (-.4,1);
	  \draw (.4,0) to (.4,1);
	  \draw (0,0) to (0,-1);
	\end{pgfonlayer}
    \end{tikzpicture}\end{aligned}\right]
    \right) = 
      \vcenter{\hbox{\begin{tikzpicture}[scale=0.75]
	\begin{pgfonlayer}{nodelayer}
		\node [style=none] (0) at (0.25, 1) {};
		\ground{1}{-0.25, 0.75};
		\node [style=none] (2) at (0.25, 0) {};
		\node [style=box] (3) at (0, 0) {$\;\;Ff\;\;$};
		\node [style=none] (4) at (-0.25, 0) {};
		\node [style=none] (5) at (0, -0.75) {};
	\end{pgfonlayer}
	\begin{pgfonlayer}{edgelayer}
		\draw (5.center) to (3);
		\draw (2.center) to (0.center);
		\draw (1) to (4.center);
	\end{pgfonlayer}
      \end{tikzpicture}}} 
  \]
  on morphisms. This map is well-defined and injective by
  Definition~\ref{def:environment}(b), and is surjective by
  Definition~\ref{def:environment}(c). It preserves composition and
  monoidal structure
  \begin{align*}
  & \xi \left( 
    \left[
    \begin{aligned}\begin{tikzpicture}[scale=0.5]
	\begin{pgfonlayer}{nodelayer}
	  \node [box] (g) at (0,0) {$\;\;g\;\;$};
	\end{pgfonlayer}
	\begin{pgfonlayer}{edgelayer}
	  \draw (-.4,0) to (-.4,1);
	  \draw (.4,0) to (.4,1);
	  \draw (0,0) to (0,-1);
	\end{pgfonlayer}
    \end{tikzpicture}\end{aligned}\right]
    \after
 \left[
    \begin{aligned}\begin{tikzpicture}[scale=0.5]
	\begin{pgfonlayer}{nodelayer}
	  \node [box] (f) at (0,0) {$\;\;f\;\;$};
	\end{pgfonlayer}
	\begin{pgfonlayer}{edgelayer}
	  \draw (-.4,0) to (-.4,1);
	  \draw (.4,0) to (.4,1);
	  \draw (0,0) to (0,-1);
	\end{pgfonlayer}
    \end{tikzpicture}\end{aligned}\right]
    \right)
    \; = \;
    \vcenter{\hbox{\begin{tikzpicture}[scale=0.75]
	\begin{pgfonlayer}{nodelayer}
		\node [style=none] (0) at (1, 1.75) {};
		\ground{1}{0, 1.5};
		\node [style=none] (2) at (1, 0.75) {};
		\node [style=box] (3) at (0.5, 0.75) {$\;\;\;Fg\;\;\;$};
		\node [style=none] (4) at (0, 0.75) {};
		\ground{5}{-0.5, 0.25};
		\node [style=none] (6) at (0.5, -0.5) {};
		\node [style=box] (7) at (0, -0.5) {$\;\;\;Ff\;\;\;$};
		\node [style=none] (8) at (-0.5, -0.5) {};
		\node [style=none] (9) at (0, -1.25) {};
	\end{pgfonlayer}
	\begin{pgfonlayer}{edgelayer}
		\draw (2.center) to (0.center);
		\draw (1) to (4.center);
		\draw (9.center) to (7);
		\draw [in=90, out=270, looseness=0.75] (3) to (6.center);
		\draw (5) to (8.center);
	\end{pgfonlayer}
      \end{tikzpicture}}}
    \; = \;
    \xi \left( 
    \left[
    \begin{aligned}\begin{tikzpicture}[scale=0.5]
	\begin{pgfonlayer}{nodelayer}
	  \node [box] (g) at (0,0) {$\;\;g\;\;$};
	\end{pgfonlayer}
	\begin{pgfonlayer}{edgelayer}
	  \draw (-.4,0) to (-.4,1);
	  \draw (.4,0) to (.4,1);
	  \draw (0,0) to (0,-1);
	\end{pgfonlayer}
    \end{tikzpicture}\end{aligned}\right]
    \right) \after \xi \left(
 \left[
    \begin{aligned}\begin{tikzpicture}[scale=0.5]
	\begin{pgfonlayer}{nodelayer}
	  \node [box] (f) at (0,0) {$\;\;f\;\;$};
	\end{pgfonlayer}
	\begin{pgfonlayer}{edgelayer}
	  \draw (-.4,0) to (-.4,1);
	  \draw (.4,0) to (.4,1);
	  \draw (0,0) to (0,-1);
	\end{pgfonlayer}
    \end{tikzpicture}\end{aligned}\right] 
    \right)
   \\
   & \xi \left( 
 \left[
    \begin{aligned}\begin{tikzpicture}[scale=0.5]
	\begin{pgfonlayer}{nodelayer}
	  \node [box] (f) at (0,0) {$\;\;f\;\;$};
	\end{pgfonlayer}
	\begin{pgfonlayer}{edgelayer}
	  \draw (-.4,0) to (-.4,1);
	  \draw (.4,0) to (.4,1);
	  \draw (0,0) to (0,-1);
	\end{pgfonlayer}
    \end{tikzpicture}\end{aligned}\right]
   \tensor
 \left[
    \begin{aligned}\begin{tikzpicture}[scale=0.5]
	\begin{pgfonlayer}{nodelayer}
	  \node [box] (g) at (0,0) {$\;\;g\;\;$};
	\end{pgfonlayer}
	\begin{pgfonlayer}{edgelayer}
	  \draw (-.4,0) to (-.4,1);
	  \draw (.4,0) to (.4,1);
	  \draw (0,0) to (0,-1);
	\end{pgfonlayer}
    \end{tikzpicture}\end{aligned}\right]
   \right)
    \; = \;
    \vcenter{\hbox{\begin{tikzpicture}[scale=0.75]
	\begin{pgfonlayer}{nodelayer}
		\node [style=none] (0) at (1.25, 1) {};
		\node [style=none] (1) at (-0.25, 1) {};
		\ground{2}{0.75, 0.75};
		\ground{3}{-0.75, 0.75};
		\node [style=none] (4) at (1.25, 0) {};
		\node [style=box] (5) at (1, 0) {$\;\;Fg\vphantom{f}\;\;$};
		\node [style=none] (6) at (0.75, 0) {};
		\node [style=none] (7) at (-0.25, 0) {};
		\node [style=box] (8) at (-0.5, 0) {$\;\;Ff\;\;$};
		\node [style=none] (9) at (-0.75, 0) {};
		\node [style=none] (10) at (1, -0.75) {};
		\node [style=none] (11) at (-0.5, -0.75) {};
	\end{pgfonlayer}
	\begin{pgfonlayer}{edgelayer}
		\draw (4.center) to (0.center);
		\draw (10.center) to (5);
		\draw (2) to (6.center);
		\draw (11.center) to (8);
		\draw (7.center) to (1.center);
		\draw (3) to (9.center);
	\end{pgfonlayer}
      \end{tikzpicture}}}        
    \; = \;
    \xi \left( 
 \left[
    \begin{aligned}\begin{tikzpicture}[scale=0.5]
	\begin{pgfonlayer}{nodelayer}
	  \node [box] (f) at (0,0) {$\;\;f\;\;$};
	\end{pgfonlayer}
	\begin{pgfonlayer}{edgelayer}
	  \draw (-.4,0) to (-.4,1);
	  \draw (.4,0) to (.4,1);
	  \draw (0,0) to (0,-1);
	\end{pgfonlayer}
    \end{tikzpicture}\end{aligned}\right]
    \right) \tensor \xi \left(
 \left[
    \begin{aligned}\begin{tikzpicture}[scale=0.5]
	\begin{pgfonlayer}{nodelayer}
	  \node [box] (g) at (0,0) {$\;\;g\;\;$};
	\end{pgfonlayer}
	\begin{pgfonlayer}{edgelayer}
	  \draw (-.4,0) to (-.4,1);
	  \draw (.4,0) to (.4,1);
	  \draw (0,0) to (0,-1);
	\end{pgfonlayer}
    \end{tikzpicture}\end{aligned}\right] 
     \right)
 \end{align*}
  because $F$ is a strict monoidal functor and by Definition~\ref{def:environment}(a). It also clearly preserves identities. Thus $\xi$ is an isomorphism of monoidal categories.
\end{proof}


Next we consider necessary conditions for a category to be of the form $\CP(\cat{C})$.
If $\cat{C}$ is a dagger braided monoidal category, there is a canonical strict monoidal functor $F \colon \cat{C} \to \CP(\cat{C})$, given by $A \mapsto A$ on objects and by $f \mapsto [\lambda^\dag f]$ on morphisms. 

\begin{theorem}\label{thm:environmentnecessary}
  Let $\cat{C}$ be a dagger braided monoidal category. 
  Then $\cat{\widehat{C}}=\CP(\cat{C})$ with the canonical functor $F \colon \cat{C} \to \cat{\widehat{C}}$ and 
  $\top_A = [\rho_A^\dag] \in \cat{\widehat{C}}(A,I)$
  is an environment structure. 
\end{theorem}
\begin{proof}
  Definition~\ref{def:environment}(a) and (c) are automatically satisfied, using Lemma~\ref{lem:ancilla} and the fact that $\lambda_I=\rho_I$. Let $f \in \cat{C}(A,X \otimes B)$ and $g \in \cat{C}(A,Y \otimes B)$.
  Unfolding definitions, we see that
  \[
    \vcenter{\hbox{\begin{tikzpicture}[scale=0.75]
	\begin{pgfonlayer}{nodelayer}
		\node [style=none] (0) at (0.25, 1) {};
		\ground{1}{-0.25, 0.75};
		\node [style=none] (2) at (0.25, 0) {};
		\node [style=box] (3) at (0, 0) {$\;\;Ff\;\;$};
		\node [style=none] (4) at (-0.25, 0) {};
		\node [style=none] (5) at (0, -0.75) {};
	\end{pgfonlayer}
	\begin{pgfonlayer}{edgelayer}
		\draw (5.center) to (3);
		\draw (2.center) to (0.center);
		\draw (1) to (4.center);
	\end{pgfonlayer}
      \end{tikzpicture}}} \;=\;
      \vcenter{\hbox{\begin{tikzpicture}[scale=0.75]
	\begin{pgfonlayer}{nodelayer}
		\node [style=none] (0) at (0.25, 1) {};
		\ground{1}{-0.25, 0.75};
		\node [style=none] (2) at (0.25, 0) {};
		\node [style=box] (3) at (0, 0) {$\;\;Fg\vphantom{f}\;\;$};
		\node [style=none] (4) at (-0.25, 0) {};
		\node [style=none] (5) at (0, -0.75) {};
	\end{pgfonlayer}
	\begin{pgfonlayer}{edgelayer}
		\draw (5.center) to (3);
		\draw (2.center) to (0.center);
		\draw (1) to (4.center);
	\end{pgfonlayer}
      \end{tikzpicture}}}
  \]
  in $\widehat{\cat{C}}$ comes down to 
  $f \sim g$ in $\cat{C}$.
  Hence Definition~\ref{def:environment}(b) is satisfied.
\end{proof}


Finally, we study when environment structures are faithful. For faithful environment structures, we may replace the functor $F$ by its image $\cat{D}$, which is a monoidal subcategory of $\CP(\cat{C})$. This is how Definition~\ref{def:environment} relates to its namesake that was previously defined for the case when $\cat{C}$ is compact~\cite{coeckeperdrix:channels}.\footnote{\todo{The environment structures of~\cite{coeckeheunen:qpl} are precisely the faithful ones of Definition~\ref{def:environment}; Definition~\ref{def:doubling} below also differs from~\cite{coeckeheunen:qpl} to match this.}} 
\todo{But if $\cat{C}$ is not compact, it is not clear whether $\cat{D}$ has a dagger, and thus whether Theorem~\ref{thm:environmentnecessary} would still hold.}
To make the comparison, for the rest of this article we will assume that $\cat{C}$ is compact.
%
We may think of $\cat{D}$ as a \emph{double} of $\cat{C}$. The next definition makes this precise.

\begin{definition}\label{def:doubling}
  A dagger monoidal category satisfies the \emph{doubling axiom} when 
  \[
    \vcenter{\hbox{\begin{tikzpicture}[scale=0.75]
	\begin{pgfonlayer}{nodelayer}
		\node [style=none] (0) at (0.6, 0.75) {};
		\node [style=box] (1) at (0.6, 0) {$\;\;f^\dag\;\;$};
		\node [style=none] (2) at (0.6, -0.75) {};
		\node [style=none] (3) at (-0.6, 0.75) {};
		\node [style=box] (4) at (-0.6, 0) {$\;\;f\vphantom{^\dag}\;\;$};
		\node [style=none] (5) at (-0.6, -0.75) {};
	\end{pgfonlayer}
	\begin{pgfonlayer}{edgelayer}
		\draw (0.center) to (1);
		\draw (2.center) to (1);
		\draw (3.center) to (4);
		\draw (5.center) to (4);
	\end{pgfonlayer}
      \end{tikzpicture}}}
    \;=\;
    \vcenter{\hbox{\begin{tikzpicture}[scale=0.75]
	\begin{pgfonlayer}{nodelayer}
		\node [style=none] (0) at (0.6, 0.75) {};
		\node [style=box] (1) at (0.6, 0) {$\;\;g^\dag\vphantom{f}\;\;$};
		\node [style=none] (2) at (0.6, -0.75) {};
		\node [style=none] (3) at (-0.6, 0.75) {};
		\node [style=box] (4) at (-0.6, 0) {$\;\;g\vphantom{f^\dag}\;\;$};
		\node [style=none] (5) at (-0.6, -0.75) {};
	\end{pgfonlayer}
	\begin{pgfonlayer}{edgelayer}
		\draw (0.center) to (1);
		\draw (2.center) to (1);
		\draw (3.center) to (4);
		\draw (5.center) to (4);
	\end{pgfonlayer}
      \end{tikzpicture}}}
    \vcenter{\hbox{}}
   \quad\Longleftrightarrow\quad
    \vcenter{\hbox{\begin{tikzpicture}[scale=0.75]
	\begin{pgfonlayer}{nodelayer}
		\node [style=none] (0) at (0, 0.75) {};
		\node [style=box] (1) at (0, 0) {$\;\;f\;\;$};
		\node [style=none] (2) at (0, -0.75) {};
	\end{pgfonlayer}
	\begin{pgfonlayer}{edgelayer}
		\draw (0.center) to (1);
		\draw (2.center) to (1);
	\end{pgfonlayer}
      \end{tikzpicture}}}
   \;=\;
    \vcenter{\hbox{\begin{tikzpicture}[scale=0.75]
	\begin{pgfonlayer}{nodelayer}
		\node [style=none] (0) at (0, 0.75) {};
		\node [style=box] (1) at (0, 0) {$\;\;g\vphantom{f}\;\;$};
		\node [style=none] (2) at (0, -0.75) {};
	\end{pgfonlayer}
	\begin{pgfonlayer}{edgelayer}
		\draw (0.center) to (1);
		\draw (2.center) to (1);
	\end{pgfonlayer}
      \end{tikzpicture}}}
 \]
  for all parallel morphisms $f$ and $g$.
\end{definition}


\begin{proposition}\label{prop:faithfuldoubling}
  If $\cat{C}$ is a dagger compact category, then the canonical functor $F \colon \cat{C} \to \CP(\cat{C})$ is faithful if and only if $\cat{C}$ satisfies the doubling axiom.
\end{proposition}
\begin{proof}
  Let $f,g \in \cat{C}(A,B)$, and suppose that $F(f)=F(g)$, that is, $\lambda^\dag f \sim \lambda^\dag g$. By Corollary~\ref{cor:CPwithouteqrel} this is equivalent to,
  \[
    \begin{aligned}\begin{tikzpicture}[scale=0.75]
	\begin{pgfonlayer}{nodelayer}
		\node [style=none] (0) at (1, 1.2) {};
		\node [style=none] (1) at (0, 1.2) {};
		\node [style=none] (2) at (1, 0.5) {};
		\node [style=none] (3) at (0, 0.5) {};
		\node [style=box] (4) at (0, 0.5) {$\;\;f^\dag\;\;$};
		\node [style=none] (5) at (-0.25, 0.5) {};
		\node [style=none] (6) at (1, -0.5) {};
		\node [style=none] (7) at (0, -0.5) {};
		\node [style=box] (8) at (0, -0.5) {$\;\;f\phantom{^\dag}\;\;$};
		\node [style=none] (9) at (-0.25, -0.5) {};
		\node [style=none] (10) at (1, -1.2) {};
		\node [style=none] (11) at (0, -1.2) {};
	\end{pgfonlayer}
	\begin{pgfonlayer}{edgelayer}
		\draw [in=90, out=270, looseness=1.25] (3.center) to (6.south);
		\draw (10) to (6);
		\draw (4) to (1.center);
		\draw [in=90, out=-90, looseness=1.25] (2.north) to (7.center);
		\draw (2) to (0);
		\draw (8) to (11.center);
	\end{pgfonlayer}\end{tikzpicture}\end{aligned}
	=
    \begin{aligned}\begin{tikzpicture}[scale=0.75]
	\begin{pgfonlayer}{nodelayer}
		\node [style=none] (0) at (1, 1.2) {};
		\node [style=none] (1) at (0, 1.2) {};
		\node [style=none] (2) at (1, 0.5) {};
		\node [style=none] (3) at (0, 0.5) {};
		\node [style=box] (4) at (0, 0.5) {$\;\;g^\dag\;\;$};
		\node [style=none] (5) at (-0.25, 0.5) {};
		\node [style=none] (6) at (1, -0.5) {};
		\node [style=none] (7) at (0, -0.5) {};
		\node [style=box] (8) at (0, -0.5) {$\;\;g\phantom{^\dag}\;\;$};
		\node [style=none] (9) at (-0.25, -0.5) {};
		\node [style=none] (10) at (1, -1.2) {};
		\node [style=none] (11) at (0, -1.2) {};
	\end{pgfonlayer}
	\begin{pgfonlayer}{edgelayer}
		\draw [in=90, out=270, looseness=1.25] (3.center) to (6.south);
		\draw (10) to (6);
		\draw (4) to (1.center);
		\draw [in=90, out=-90, looseness=1.25] (2.north) to (7.center);
		\draw (2) to (0);
		\draw (8) to (11.center);
	\end{pgfonlayer}\end{tikzpicture}\end{aligned}
  \]
  Naturality of the braiding shows this to be equivalent to $f \otimes f^\dag = g \otimes g^\dag$.
  This implies $f=g$ precisely when $\cat{C}$ satisfies the doubling axiom.
\end{proof}

Recall that a dagger compact category $\cat{C}$ satisfies the
so-called \emph{preparation-state agreement
axiom}~\cite{coecke:projective} when $f f^\dag = g g^\dag$ implies $f=g$ for all morphisms $f$ and $g$ with domain $I$:
\[
    \vcenter{\hbox{\begin{tikzpicture}[scale=0.75]
	\begin{pgfonlayer}{nodelayer}
		\node [style=none] (0) at (0.8, 1.25) {};
		\node [style=box] (1a) at (0.8, 0.5) {$\;\;f\phantom{^\dag}\;\;$};
		\node [style=box] (1) at (0.8, -0.5) {$\;\;f^\dag\;\;$};
		\node [style=none] (2) at (0.8, -1.25) {};
	\end{pgfonlayer}
	\begin{pgfonlayer}{edgelayer}
		\draw (0.center) to (1a);
		\draw (2.center) to (1);
	\end{pgfonlayer}
      \end{tikzpicture}}}
    \;=\;
    \vcenter{\hbox{\begin{tikzpicture}[scale=0.75]
	\begin{pgfonlayer}{nodelayer}
		\node [style=none] (0) at (0.8, 1.25) {};
		\node [style=box] (1a) at (0.8, 0.5) {$\;\;\vphantom{f^\dag}g\phantom{^\dag}\;\;$};
		\node [style=box] (1) at (0.8, -0.5) {$\;\;\vphantom{f^\dag}g^\dag\;\;$};
		\node [style=none] (2) at (0.8, -1.25) {};
	\end{pgfonlayer}
	\begin{pgfonlayer}{edgelayer}
		\draw (0.center) to (1a);
		\draw (2.center) to (1);
	\end{pgfonlayer}
      \end{tikzpicture}}}
  \quad\Longleftrightarrow\quad
    \vcenter{\hbox{\begin{tikzpicture}[scale=0.75]
	\begin{pgfonlayer}{nodelayer}
		\node [style=none] (0) at (0.8, 1.25) {};
		\node [style=box] (1a) at (0.8, 0.5) {$\;\;f\vphantom{^\dag}\;\;$};
	\end{pgfonlayer}
	\begin{pgfonlayer}{edgelayer}
		\draw (0.center) to (1a);
	\end{pgfonlayer}
      \end{tikzpicture}}}
    \;=\;
    \vcenter{\hbox{\begin{tikzpicture}[scale=0.75]
	\begin{pgfonlayer}{nodelayer}
		\node [style=none] (0) at (0.8, 1.25) {};
		\node [style=box] (1a) at (0.8, 0.5) {$\;\;\vphantom{f^\dag}g\;\;$};
	\end{pgfonlayer}
	\begin{pgfonlayer}{edgelayer}
		\draw (0.center) to (1a);
	\end{pgfonlayer}
      \end{tikzpicture}}}
\]
So if $\cat{C}$ is a dagger compact category, then $\CPM(\cat{C})$
satisfies the preparation-state agreement axiom when
\[
    \vcenter{\hbox{\begin{tikzpicture}[scale=0.75]
	\begin{pgfonlayer}{nodelayer}
		\node [style=none] (0) at (-0.75, 1) {};
		\node [style=none] (1) at (0.75, 1) {};
		\node [style=box] (2) at (-0.75, 0) {$\;\;f_*\vphantom{^\dag}\;\;$};
		\node [style=box] (3) at (0.75, 0) {$\;\;f\phantom{{}_*^\dag}\;\;$};
		\node [style=none] (0a) at (-0.75, -3) {};
		\node [style=none] (1a) at (0.75, -3) {};
		\node [style=box] (2a) at (-0.75, -2) {$\;\;f_*^\dag\;\;$};
		\node [style=box] (3a) at (0.75, -2) {$\;\;f^\dag\;\;$};
	\end{pgfonlayer}
	\begin{pgfonlayer}{edgelayer}
		\draw (0) to (2);
		\draw [bend left=90, looseness=1.25] (3) to (2);
		\draw (1) to (3);
		\draw (0a) to (2a);
		\draw [bend right=90, looseness=1.25] (3a) to (2a);
		\draw (1a) to (3a);
	\end{pgfonlayer}
      \end{tikzpicture}}}
  \;=\;
    \vcenter{\hbox{\begin{tikzpicture}[scale=0.75]
	\begin{pgfonlayer}{nodelayer}
		\node [style=none] (0) at (-0.75, 1) {};
		\node [style=none] (1) at (0.75, 1) {};
		\node [style=box] (2) at (-0.75, 0) {$\;\;g_*\vphantom{^\dag}\;\;$};
		\node [style=box] (3) at (0.75, 0) {$\;\;g\phantom{{}_*^\dag}\;\;$};
		\node [style=none] (0a) at (-0.75, -3) {};
		\node [style=none] (1a) at (0.75, -3) {};
		\node [style=box] (2a) at (-0.75, -2) {$\;\;g_*^\dag\;\;$};
		\node [style=box] (3a) at (0.75, -2) {$\;\;g^\dag\;\;$};
	\end{pgfonlayer}
	\begin{pgfonlayer}{edgelayer}
		\draw (0) to (2);
		\draw [bend left=90, looseness=1.25] (3) to (2);
		\draw (1) to (3);
		\draw (0a) to (2a);
		\draw [bend right=90, looseness=1.25] (3a) to (2a);
		\draw (1a) to (3a);
	\end{pgfonlayer}
      \end{tikzpicture}}}
 \quad\Longleftrightarrow\quad
    \vcenter{\hbox{\begin{tikzpicture}[scale=0.75]
	\begin{pgfonlayer}{nodelayer}
		\node [style=none] (0) at (-0.75, 1) {};
		\node [style=none] (1) at (0.75, 1) {};
		\node [style=box] (2) at (-0.75, 0) {$\;\;f_*\phantom{^\dag}\;\;$};
		\node [style=box] (3) at (0.75, 0) {$\;\;f\phantom{_*^\dag}\;\;$};
	\end{pgfonlayer}
	\begin{pgfonlayer}{edgelayer}
		\draw (0) to (2);
		\draw [bend left=90, looseness=1.25] (3) to (2);
		\draw (1) to (3);
	\end{pgfonlayer}
      \end{tikzpicture}}}
   \;=\;
    \vcenter{\hbox{\begin{tikzpicture}[scale=0.75]
	\begin{pgfonlayer}{nodelayer}
		\node [style=none] (0) at (-0.75, 1) {};
		\node [style=none] (1) at (0.75, 1) {};
		\node [style=box] (2) at (-0.75, 0) {$\;\;g_*\phantom{^\dag}\;\;$};
		\node [style=box] (3) at (0.75, 0) {$\;\;g\phantom{_*^\dag}\;\;$};
	\end{pgfonlayer}
	\begin{pgfonlayer}{edgelayer}
		\draw (0) to (2);
		\draw [bend left=90, looseness=1.25] (3) to (2);
		\draw (1) to (3);
	\end{pgfonlayer}
      \end{tikzpicture}}}
\]
for all morphisms $f$ and $g$ with a common domain. 
The preparation-state agreement axiom (for dagger compact $\cat{C}$) is equivalent to the
doubling axiom (for $\cat{C})$\todo{~\cite[Section~3]{coecke:mix}}.
\todo{However, the doubling axiom has an advantage over the preparation-state agreement axiom in that Proposition~\ref{prop:faithfuldoubling} can work when $\cat{C}$ is not compact, too, with the condition of Corollary~\ref{cor:CPwithouteqrel}. Moreover}, the doubling axiom is stronger than the 
preparation-state axiom in the sense of the following proposition. 

\begin{proposition}
 If a dagger compact category $\cat{C}$ satisfies the doubling axiom,
 then $\CPM(\cat{C})$ satisfies the preparation-state agreement axiom.
\end{proposition}
\begin{proof}
  Let $f$ and $g$ be morphisms with a common domain. Elementary
  graphical manipulations yield:
  \[
    \vcenter{\hbox{\begin{tikzpicture}[scale=0.75]
	\begin{pgfonlayer}{nodelayer}
		\node [style=none] (0) at (-0.75, 1) {};
		\node [style=none] (1) at (0.75, 1) {};
		\node [style=box] (2) at (-0.75, 0) {$\;\;f_*\vphantom{^\dag}\;\;$};
		\node [style=box] (3) at (0.75, 0) {$\;\;f\phantom{{}_*^\dag}\;\;$};
		\node [style=none] (0a) at (-0.75, -3) {};
		\node [style=none] (1a) at (0.75, -3) {};
		\node [style=box] (2a) at (-0.75, -2) {$\;\;f_*^\dag\;\;$};
		\node [style=box] (3a) at (0.75, -2) {$\;\;f^\dag\;\;$};
	\end{pgfonlayer}
	\begin{pgfonlayer}{edgelayer}
		\draw (0) to (2);
		\draw [bend left=90, looseness=1.25] (3) to (2);
		\draw (1) to (3);
		\draw (0a) to (2a);
		\draw [bend right=90, looseness=1.25] (3a) to (2a);
		\draw (1a) to (3a);
	\end{pgfonlayer}
      \end{tikzpicture}}}
  \;=\;
    \vcenter{\hbox{\begin{tikzpicture}[scale=0.75]
	\begin{pgfonlayer}{nodelayer}
		\node [style=none] (0) at (-0.75, 1) {};
		\node [style=none] (1) at (0.75, 1) {};
		\node [style=box] (2) at (-0.75, 0) {$\;\;g_*\vphantom{^\dag}\;\;$};
		\node [style=box] (3) at (0.75, 0) {$\;\;g\phantom{{}_*^\dag}\;\;$};
		\node [style=none] (0a) at (-0.75, -3) {};
		\node [style=none] (1a) at (0.75, -3) {};
		\node [style=box] (2a) at (-0.75, -2) {$\;\;g_*^\dag\;\;$};
		\node [style=box] (3a) at (0.75, -2) {$\;\;g^\dag\;\;$};
	\end{pgfonlayer}
	\begin{pgfonlayer}{edgelayer}
		\draw (0) to (2);
		\draw [bend left=90, looseness=1.25] (3) to (2);
		\draw (1) to (3);
		\draw (0a) to (2a);
		\draw [bend right=90, looseness=1.25] (3a) to (2a);
		\draw (1a) to (3a);
	\end{pgfonlayer}
      \end{tikzpicture}}}
 \quad\Longleftrightarrow\quad
    \vcenter{\hbox{\begin{tikzpicture}[scale=0.75]
	\begin{pgfonlayer}{nodelayer}
		\node [style=none] (0) at (-0.75, 1) {};
		\node [style=none] (1) at (0.75, 1) {};
		\node [style=box] (2) at (-0.75, 0) {$\;\;f\phantom{^\dag}\;\;$};
		\node [style=box] (3) at (0.75, 0) {$\;\;f\phantom{^\dag}\;\;$};
		\node [style=none] (0a) at (-0.75, -2) {};
		\node [style=none] (1a) at (0.75, -2) {};
		\node [style=box] (2a) at (-0.75, -1) {$\;\;f^\dag\;\;$};
		\node [style=box] (3a) at (0.75, -1) {$\;\;f^\dag\;\;$};
	\end{pgfonlayer}
	\begin{pgfonlayer}{edgelayer}
		\draw (0) to (2);
                \draw (2a) to (2);
                \draw (3a) to (3);
		\draw (1) to (3);
		\draw (0a) to (2a);
		\draw (1a) to (3a);
	\end{pgfonlayer}
      \end{tikzpicture}}}
  \;=\;
    \vcenter{\hbox{\begin{tikzpicture}[scale=0.75]
	\begin{pgfonlayer}{nodelayer}
		\node [style=none] (0) at (-0.75, 1) {};
		\node [style=none] (1) at (0.75, 1) {};
		\node [style=box] (2) at (-0.75, 0) {$\;\;g\phantom{^\dag}\;\;$};
		\node [style=box] (3) at (0.75, 0) {$\;\;g\phantom{^\dag}\;\;$};
		\node [style=none] (0a) at (-0.75, -2) {};
		\node [style=none] (1a) at (0.75, -2) {};
		\node [style=box] (2a) at (-0.75, -1) {$\;\;g^\dag\;\;$};
		\node [style=box] (3a) at (0.75, -1) {$\;\;g^\dag\;\;$};
	\end{pgfonlayer}
	\begin{pgfonlayer}{edgelayer}
		\draw (0) to (2);
                \draw (2a) to (2);
                \draw (3a) to (3);
		\draw (1) to (3);
		\draw (0a) to (2a);
		\draw (1a) to (3a);
	\end{pgfonlayer}
      \end{tikzpicture}}}
  \]
  By the doubling axiom for $\cat{C}$, the latter is equivalent to
  $f f^\dag = g g^\dag$. Finally, by using graphical manipulation
  again, we obtain:
 \[
   \vcenter{\hbox{\begin{tikzpicture}[scale=0.75]
	\begin{pgfonlayer}{nodelayer}
		\node [style=none] (1) at (0.75, 1) {};
		\node [style=box] (3) at (0.75, 0) {$\;\;f\phantom{^\dag}\;\;$};
		\node [style=none] (1a) at (0.75, -2) {};
		\node [style=box] (3a) at (0.75, -1) {$\;\;f^\dag\;\;$};
	\end{pgfonlayer}
	\begin{pgfonlayer}{edgelayer}
               \draw (3a) to (3);
		\draw (1) to (3);
		\draw (1a) to (3a);
	\end{pgfonlayer}
      \end{tikzpicture}}}
  \;=\;
    \vcenter{\hbox{\begin{tikzpicture}[scale=0.75]
	\begin{pgfonlayer}{nodelayer}
		\node [style=none] (1) at (0.75, 1) {};
		\node [style=box] (3) at (0.75, 0) {$\;\;g\phantom{^\dag}\;\;$};
		\node [style=none] (1a) at (0.75, -2) {};
		\node [style=box] (3a) at (0.75, -1) {$\;\;g^\dag\;\;$};
	\end{pgfonlayer}
	\begin{pgfonlayer}{edgelayer}
               \draw (3a) to (3);
		\draw (1) to (3);
		\draw (1a) to (3a);
	\end{pgfonlayer}
      \end{tikzpicture}}}
 \quad\Longleftrightarrow\quad
    \vcenter{\hbox{\begin{tikzpicture}[scale=0.75]
	\begin{pgfonlayer}{nodelayer}
		\node [style=none] (0) at (-0.75, 1) {};
		\node [style=none] (1) at (0.75, 1) {};
		\node [style=box] (2) at (-0.75, 0) {$\;\;f_*\phantom{^\dag}\;\;$};
		\node [style=box] (3) at (0.75, 0) {$\;\;f\phantom{_*^\dag}\;\;$};
	\end{pgfonlayer}
	\begin{pgfonlayer}{edgelayer}
		\draw (0) to (2);
		\draw [bend left=90, looseness=1.25] (3) to (2);
		\draw (1) to (3);
	\end{pgfonlayer}
      \end{tikzpicture}}}
   \;=\;
    \vcenter{\hbox{\begin{tikzpicture}[scale=0.75]
	\begin{pgfonlayer}{nodelayer}
		\node [style=none] (0) at (-0.75, 1) {};
		\node [style=none] (1) at (0.75, 1) {};
		\node [style=box] (2) at (-0.75, 0) {$\;\;g_*\phantom{^\dag}\;\;$};
		\node [style=box] (3) at (0.75, 0) {$\;\;g\phantom{_*^\dag}\;\;$};
	\end{pgfonlayer}
	\begin{pgfonlayer}{edgelayer}
		\draw (0) to (2);
		\draw [bend left=90, looseness=1.25] (3) to (2);
		\draw (1) to (3);
	\end{pgfonlayer}
      \end{tikzpicture}}}
 \]
  Hence $\CPM(\cat{C})$ satisfies the state-preparation agreement axiom.
\end{proof}

\bibliographystyle{plain}
\bibliography{cp}

\end{document}